\documentclass{article}[12pt]

\usepackage{arxiv}
\usepackage{quiver}

\usepackage[utf8]{inputenc} %
\usepackage[T1]{fontenc}    %
\usepackage[colorlinks = true, linkcolor = purple, citecolor = purple, urlcolor = purple]{hyperref}       %
\usepackage{url}            %
\usepackage{booktabs}       %
\usepackage{amsfonts}       %
\usepackage{nicefrac}       %
\usepackage{microtype}      %
\usepackage{lipsum}		%
\usepackage[normalem]{ulem}
\usepackage{stackengine}
\usepackage{color,graphicx} %
\usepackage{scalerel} %
\usepackage{mathtools,amsmath,amsthm,amscd,amssymb} %
\usepackage[font={up}]{caption}%
\usepackage{subcaption} %
\usepackage{float}
\usepackage{colortbl}
\usepackage{soul} %
\usepackage{tabularx}
\usepackage{multicol} %
\usepackage{doi}
\usepackage{verbatimbox}
\usepackage[shortlabels]{enumitem}
\usepackage{blindtext}
\usepackage{wrapfig}
\usepackage{adjustbox}
\usepackage[style=authoryear, backend=biber, natbib=true]{biblatex}
\usepackage{algorithm}%
\usepackage{algorithmicx}%
\usepackage{algpseudocode}%

\newtheoremstyle{exampstyle}
  {2pt} %
  {.5pt} %
  {} %
  {} %
  {\bfseries} %
  {.} %
  {.5em} %
  {} %

\newtheorem{theorem}{Theorem}[section]

\newtheorem{corollary}[theorem]{Corollary}

\newtheorem{lemma}[theorem]{Lemma}

\newtheorem{proposition}[theorem]{Proposition}

\theoremstyle{definition}
\newtheorem{definition}[theorem]{Definition}
\newtheorem{example}[theorem]{Example}

\theoremstyle{remark}

\newtheorem*{remark}{Remark}

\newcommand{\F}{\mathbb{F}}
\newcommand{\I}{\mathbb{I}}
\newcommand{\N}{\mathbb{N}}
\newcommand{\Q}{\mathbb{Q}}
\newcommand{\R}{\mathbb{R}}
\newcommand{\Z}{\mathbb{Z}}

\newcommand{\B}{\mathcal{B}}
\newcommand{\C}{\mathcal{C}}
\newcommand{\E}{\mathcal{E}}
\newcommand{\D}{\mathcal{D}}
\renewcommand{\P}{\mathcal{P}}

\newcommand{\U}{\mathcal{U}}
\newcommand{\T}{\mathcal{T}}

\newcommand{\cR}{\mathcal{R}}
\newcommand{\cS}{\mathcal{S}}
\newcommand{\cT}{\mathcal{T}}

\newcommand{\q}{\mathbf{q}}

\DeclareMathOperator{\rk}{rk}

\DeclarePairedDelimiter{\ceil}{\lceil}{\rceil}
\DeclarePairedDelimiter\floor{\lfloor}{\rfloor}
\DeclarePairedDelimiter\anglebracket{\langle}{\rangle}

\newcommand{\im}{\text{Im}}

\newcommand{\End}{\text{End}}

\newcommand{\tuple}[3][1]{#2_{#1}, \dots , #2_{#3}}

\newcommand{\ncoord}[2]{\anglebracket{{#1}^1, \ldots, {#1}^{#2}}}

\newcommand{\Trans}{\mathrm{Trans}}

\renewcommand{\Vec}{\mathrm{Vec}}
\newcommand{\Set}{\mathrm{Set}}
\newcommand{\Flows}{\mathrm{Flows}}

\newcommand\rightdsar[1][0.8ex]{%
  \mathbin{\hspace{0.5ex}%
  \scalebox{-1}[-1]{\stackon[-3.7pt]{\buwave{\hspace{#1}}}%
  {\buwave{\hspace{#1}}}}\mkern-2.6mu\scalebox{.7}[1.1]{$\succ$}}
} 

\newcommand\buwave[1]{\raisebox{-1.6pt}{\uwave{\raisebox{2pt}{#1}}}}

\newcommand*{\Parallelogramr}[1][]{
  \pgfpicture\pgfsetroundjoin
    \pgftransformxslant{.6}
    \pgfpathrectangle{\pgfpointorigin}{\pgfpoint{.60em}{.65em}}
    \pgfusepath{stroke,#1}
  \endpgfpicture}

\newcommand*{\Parallelograml}[1][]{
  \pgfpicture\pgfsetroundjoin
    \pgftransformxslant{-.6}
    \pgfpathrectangle{\pgfpointorigin}{\pgfpoint{.60em}{.65em}}
    \pgfusepath{stroke,#1}
  \endpgfpicture}

\newcommand{\triangled}{\raisebox{\depth}{$\bigtriangledown$}}
\newcommand{\triangleu}{\bigtriangleup}

\newcommand{\Lpl}{L_{\scaleobj{.7}{\Parallelograml}}}
\newcommand{\Lpr}{L_{\scaleobj{.7}{\Parallelogramr}}}
\newcommand{\Ltd}{L_{\bigtriangledown}}
\newcommand{\Ltu}{L_{\bigtriangleup}}

\newcommand{\e}{\varepsilon}
\renewcommand{\phi}{\varphi}

\newcommand{\id}{\mathrm{Id}}
\newcommand{\Id}{\id}
\newcommand{\Int}{\mathrm{Int}}

\newcommand{\intAny}[1]{\lfloor #1 \rceil}

\title{Bounding the interleaving distance on concrete categories using a loss function}

\date{\today} 					%

\author{Astrid A. Olave \And Elizabeth Munch}
\renewcommand{\shorttitle}{}

\hypersetup{
pdftitle={},
pdfsubject={},
pdfauthor={Astrid A.~Olave, Elizabeth Munch},
pdfkeywords={},
}

\begin{document}
\maketitle

\begin{abstract}
 The interleaving distance is arguably the most widely used metric in topological data analysis (TDA) due to its applicability to a wide array of inputs of interest, such as (multiparameter) persistence modules, Reeb graphs, merge trees, and zigzag modules. 
 However, computation of the interleaving distance in the vast majority of this settings is known to be NP-hard, limiting its use in practical settings. 
 Inspired by the work of Chambers et al.~on the interleaving distance for mapper graphs, we solve a more general problem bounding the interleaving distance between generalized persistence modules on concrete categories via a loss function. 
 This loss function measures how far an \textit{assignment}, which can be thought of as an interleaving that might not commute, is from defining a true interleaving. 
 We give settings for which the loss can be computed in polynomial time, including for certain assumptions on 
 $k$-parameter persistence modules.
\end{abstract}

\setlength{\abovedisplayskip}{7pt}
\setlength{\belowdisplayskip}{7pt}

\section{Introduction}

Persistent homology lies at the core of topological data analysis \citep{Ghrist_2007, edelsbrunner2008, Oudot_pers}. 
In its simplest form, persistence arises as follows. 
Given a simplicial complex $K$ together with a filtration $\{K^i\}_{i \in \Z_{\geq 0}}$ where $K_i \subseteq K_j$ for $i \leq j$, one considers the sequence of homology groups $H(K^i)$ which give rise to a \textit{persistence module}, a family of modules over a ring indexed by $\Z_{\geq 0}$ \citep{Zomorodian2004}. 
Later, these ideas were generalized to define persistence modules in categorical terms as functors from linear orders, such as $\Z_{\geq 0}$ or $\R$, to the category $\Vec_\mathbb{F}$ of finite-dimensional vector spaces over a fixed field $\mathbb{F}$ \citep{categorification}. 
Generalizing this notion further, a functor from a poset $P$ to a category $C$ is called a \textit{generalized persistence module}  valued in $C$  indexed by $P$ \citep{Bubenik2015}. 

The \textit{interleaving distance} $d_I$ is one of the most used metrics to quantify the dissimilarity between generalized persistence modules. 
It was first introduced by \cite{2009Chazal} as an algebraic generalization of the combinatorial bottleneck distance for  persistence diagrams. 
Since it simplifies to the bottleneck distance, the interleaving distance 
for persistence modules in $\Vec_\mathbb{F}$ over $\R$, can be computed efficiently in polynomial time. 
Since then, the interleaving distance has been generalized for different posets and categories. 
When $P = \R^2$ and $C = \Vec_\F$, this gives the interleaving distance for multiparameter persistence modules \citep{lesnick2015}.
These functors can represent graph based objects by using $C = \Set$.
For this, $P=(\R,\leq)$ we have an interleaving distance for merge trees \citep{Morozov2013, Beurskens2025}, and when $P = \Int$ gives an interleaving distance for Reeb graphs \citep{deSilva2016}.
There are also notions of labeled versions of these inputs, such as labeled merge trees \citep{Munch2019, Gasparovic2024, Yan2019, Curry2022}, labeled Reeb graphs \citep{Lan2024} and formigrams \citep{Kim2023a}. 
The interleaving distance and its variants has also been studied for 
filtered spaces \citep{Blumberg2023}, 
pushforward and pullback maps \citep{Botnan2020},
zigzag persistence \citep{botnan2018,Bjerkevik2021},
and dynamic metric spaces \citep{Kim2020,Kim2023a}.
Removing the functorial object assumption means that the interleaving distance has been studied for more general categorical formulations \citep{Flow,McFaddin2023,Meehan2017,Cruz2019,Scoccola2020,Berkouk2021,Patel2018}. 

While the interleaving distance is widely used in theory, it's ability to be used in practice is limited due to its computational complexity in the vast majority of settings. 
Other than single parameter persistence modules (where the interleaving distance is the bottleneck distance), labeled merge trees, and formigrams, nearly every other construction is NP-hard to compute \citep{Bjerkevik2018,havard,Agarwal2018a}. 
Despite the fact that one will not find an efficient algorithm unless $\mathrm{P}=\mathrm{NP}$, there has been active recent progress on finding work-arounds to make computation possible \citep{Curry2022,Pegoraro2025,FarahbakhshTouli2019,2024bounding,Chambers2025}. 
In this paper, we take inspiration from the approach of  \cite{2024bounding}. 
There, the focus was on the interleaving distance for discretizations of Reeb graphs called mapper graphs. 
The interleaving distance is defined on generalized persistence modules taking value in $\Set$, and so an interleaving is a pair of natural transformations with particular properties. 
In that paper, they study collections of maps with the potential to be an interleaving but which might fail to satisfy the required commutativity conditions; these are called \emph{assignments} (see the related concept for general sheaves in \cite{Robinson2020}). 
Using these assignments, they established an upper bound on the interleaving distance between mapper graphs by defining a loss function which measures how far the input assignments are from commuting. 
That loss function was combined with a linear programming framework in \cite{Chambers2025} to find experimentally optimal interleavings between mapper graphs.

\paragraph{Our contributions}
In this paper, we extend the idea of assignments from \cite{2024bounding} to generalized persistence modules valued in a concrete category. 
Specifically following \cite{Flow}, we start with a pair of persistence modules $F,G \colon P \to C$ and a flow $\cT$ where $\T_\e = \T(\e)$ is an endomorphism of $P$, i.e.~$\T\colon [0,\infty) \to \End(P)$. 
A pair of natural transformations 
$\Phi\colon F \Rightarrow G \T_\e$ and 
$\Psi\colon G \Rightarrow F\T_\e$ 
are a $\T_\e$-interleaving of $F$ and $G$ if 
\begin{equation*}
F[p \leq \T_\e \T_\e p](p) = \Phi_p \Psi_{\T_\e p} \quad \text{and} \quad G[p \leq \T_\e \T_\e p](p) = \Psi_p \Phi_{\T_\e p}  \text{ for all } p \in P.
\end{equation*}
We set $d_I(F,G)$, their interleaving distance, to be the infimum of the $\e$'s such that there is an $\T_\e$-interleaving of $F$ and $G$. 
A collection of morphisms $\phi = \{ \phi_p\colon F(p) \rightarrow G \T_\e(p)\} $ and $\psi = \{ \psi_p\colon G(p) \rightarrow F \T_\e(p)\} $ are called an $\e$-assignment if   they do not necessarily define a $\T_\e$-interleaving. 
We define a pseudometric for elements in each $F(p)$ (sym.~each $G(p)$) called the merging distance (Def.~\ref{def_merging}). 
Using this distance, we can define a loss function generalizing that of \citep{2024bounding} in Defn.~\ref{def:loss} for a given input $\e$-assignment, ($\phi,\psi$), that measures how far is the assignment from being an interleaving.
We show that this loss, $L(\phi,\psi)$, provides an upper bound for the interleaving distance between persistence modules: $d_I(F,G) \leq \e + L(\phi,\psi)$ (Thm. \ref{th_bounding}).

Then, for reasonable additional conditions, we optimize the computation of $L(\phi,\psi)$ for persistence modules by indexed by $P$ (Thm.~\ref{thm:ControlledLossLine}) or $P^k$ (Thm. \ref{th_loss_Tk}) controlling the number of diagrams should be checked when $P$ is a complete linear order.
This makes it possible to give polynomial time algorithms for the computation for persistence modules valued in categories $C$ whose objects have underlying finite sets (Alg.~\ref{alg:computation}) or in $\Vec_\mathbb{F}$ (Alg.~\ref{alg:lpl_calculation}).
Our result opens up the potential for an algorithmic approach to bound from above the interleaving distance on commonly used constructions such as multiparameter persistence modules using this context.
However, we note that, as in the case of \citep{2024bounding,Chambers2025}, this bound will only be as good as the input assignment provided so we make no promises of tightness. 
Additional care will be needed in the future to find ways to search the space of assignments to optimize this bound; however, because we do not expect to accidentally prove $\mathrm{P}=\mathrm{NP}$ in the course of this work, any of these search algorithms will not be in $P$ otherwise the combination of the bound and the search algorithm would compute the interleaving distance in polynomial time.

\paragraph{Outline}
In Section \ref{Sec_back} we provide a background on posets, category theory and interleaving distance for generalized persistence modules. 
Section \ref{sec_loss} we define the loss function $L$ and bound the interleaving distance. 
The optimization to compute $L$ for different settings generalized persistence modules is done in Section \ref{sec_improve}. 
Finally, we discuss some implications from our results in Section \ref{sec_discussion}.

\section{Background}

\label{Sec_back}

\subsection{Posets}

We start by setting the notation and definitions for posets \citep{Harzheim, Schroder}. 
A \emph{partially ordered set (poset)} is a pair $(P,\leq)$ where $P$ is a set with a reflexive, antisymmetric and transitive relation $\leq$ on $P$. 
We say that $p,q \in P$ are \emph{comparable} if $ p \leq q $ or $q \leq p$.

A subset $S \subset P$ has an \emph{upper bound} $p$ if $p \geq s$ for all $s \in S$. 
The \emph{supremum} (or least upper bound) of $S$, if it exists, is an element $\sup S \in P$ such that for any $p$ which is an upper bound of $S$, $\sup S \leq p$. 
Similarly, for a given set $S$ we can define a \emph{lower bound} and the \emph{infimum} (or greatest lower bound) by reversing the inequalities in the previous definition. 
A poset is called  \emph{complete} if every nonempty subset that has an upper bound has a supremum and every nonempty subset that has a lower bound has infimum. 
A subset $S \subset P$ is \emph{convex} if for any $p,q \in S$ and any $r \in P$ with $p< r < q$, then $r \in S$.

We say $P$ is a \emph{dense poset} if for any $p < q$ there is an intermediate element $p<r<q$, otherwise we say $P$ is discrete. 
In the discrete case, if $p < q $ in $P$ and there is no $r \in P$ such that $p <r < q$, then $p$ is a \emph{predecessor} of $q$ and $q$ is a \emph{successor} of $p$; this is denoted $p \prec q$. 
Points $p$ and $q$ that satisfy $p \prec q$ or $p \prec q$ will also be called \emph{adjacent}. 
We will denote by $d(P)\in \N \cup \{\infty\}$ the maximum number of predecessors of any element of $P$. 

A non-empty sequence $\gamma$ of  $x_i's $ in $P$ is called a \emph{chain} if all $x_i$'s in $\gamma$ are comparable. If $P$ itself is a chain we call it a \emph{total} or \emph{linear} poset. 
    For example, the set $[0,\infty)$ with the usual order is a linear poset. 
    We can form another linear poset $[0,\infty] = [0,\infty) \cup \{ \infty\}$ by adding the element  $\infty$ to $[0,\infty)$ and setting $\infty > r$ for all $r \in [0,\infty)$

The \emph{height} of $P$, $h(P)$, is the length of the longest chain in $P$. 
If there are chains of arbitrary length then $h(P)=\infty$. 
A subset $S$  of $P$ is called an \emph{antichain} if for all $x\neq y$ in $S$, $x$ and $y$ are not comparable. 
We define the \emph{width} of $P$, $w(P)$, to be the size of the largest antichain in $P$ if such a set exists; otherwise, $w(P) = \infty$.

Given posets $P$ and $Q$, the direct product poset $P \times Q$ has elements $\anglebracket{p,q}$ for all $p \in P$ and $q \in Q$ with the order relation $\anglebracket{p,q} \leq_{P \times Q} \anglebracket{p',q'}$ if $p \leq_P p'$ and $q \leq_Q q'$. 
In particular, for $k \in  \N_{\geq 1}$ we can consider the the $k$-product poset $P^k : = P \times \cdots \times P = \bigtimes_{i=1}^k P$ which has elements $\mathbf{p} = \anglebracket{p^1, \cdots, p^k}$. We call $p^i$ the $i$-th coordinate of  $\mathbf{p}$.
Several properties of posets are inherited by their product. 
For example, the product of complete posets is complete and the product of two convex subsets is convex 

Following \cite{botnan2020decomposition}, an \emph{interval} $I$ of $P$ is a non-empty convex connected subset of $P$. 
Given  $p \subset P$ we define its \emph{open initial segment}, $(-\infty,p)$, and \emph{open final  segment}, $(p,\infty)$, by
\begin{equation*}
    (-\infty,p) = \{q \in P \mid  q < p \}, \qquad
    (p,\infty) = \{ q \in P \mid q > p\} 
\end{equation*} 

Analogously we define $(-\infty,p] = \{ q \leq p \mid q \in P\}$ and $[p,\infty)= \{ q \geq p \mid q \in P\}$ and call it the \emph{closed initial (resp. final) segment} of $p$. 
These four classes of sets are called the segments of $P$. 
The segments of $P$ as well as their intersections are intervals of $P$; for example for $p \leq q$
\begin{align*}
    [p,q] & :=  [p,\infty) \cap (-\infty,q], \text { the usual closed interval with ends } p, q \\
    (p,q) & :=  [p,\infty) \cap (-\infty,q), \text { the usual open interval with ends } p, q 
\end{align*}    
For $p,q  \in P$, we use the notation $\intAny{p,q}$ to denote any interval of the form $[p,q], (p,q), (p,q]$, or $[p,q)$. 

In a complete linear poset $T$, any interval is either a segment ($(\infty, q)$, $(\infty, q]$, $(p,\infty)$ or $ [p,\infty)$) or of the form $\intAny{p,q}$ for some $p < q \in T$. 
However that is not the case for other posets. 
The set $\{ x \mid x^2 < 5 \} \subset \Q$ is an interval of $\Q$ that can not be written as an usual interval of the form $\lfloor p,q\rceil$. 
Similarly, the set  $\{\anglebracket{0,0}, \anglebracket{0,1}, \anglebracket{1,0} \}  \subset \N \times \N$ is an interval of $\N \times \N $ but is not of the form $\lfloor p,q\rceil$.

\begin{definition}
    The \emph{order topology} of $P$ is the topology generated by the subbase of open segments of $P$, given by $\{(p,\infty) \mid p \in P\} \cup \{(-\infty, q) \mid q \in Q\}$. 
\end{definition}

For example, the usual topology of $\R$ coincides with its order topology as a poset $(\R,\leq)$. 

\begin{lemma}
    \label{lemma_compact}
    Let $T$ be a complete linear poset. 
    Any interval $[s,t] \subseteq T$ is compact in the order topology of $T$.
\end{lemma}
\begin{proof}
Let $\mathcal{U}$ be a given open cover of $[s,t]$.
Define $B$ to be the set of elements $r \in [s,t]$ such that $[s,r]$ can be covered by a finite number of elements of $\mathcal{U}$. 
The set $B$ is not empty since $s \in B$ and is bounded by $t$. 
Then, because $T$ is a complete poset, $\sup B$ exists. 
If $\sup B < t$ we can find  an interval $ (x,y) \subset U \in \mathcal{U}$ such that $\sup B \in (x,y) \subset [s,t]$. 
Then $y \in B$, but it contradicts the definition of supremum. 
Hence $t \leq \sup B$ so $[s,t]$ is compact.
    \end{proof}

\subsection{Categorical terminology}

We give basic definitions for the categorical terminology used in this paper; see \cite{riehl2017category} for details. 
A \emph{category} $C$ consists of a collection of objects, e.g.~$X,Y,Z \in C$, and for any pair of objects $X,Y$ a collection (possibly empty) of morphisms denoted $C(X,Y)$. 
We often write $f: X  \to Y$ if $f \in C(X,Y)$. We say that the \textit{source} of $f$ is $X$ and its \textit{target} is $Y$. There must be an identity morphism $\id_X \in C(X,X)$ for all objects $X \in C$. Further, there is a composition operation on morphisms $g \circ f \in C(X,Z)$ for $f:X \to Y$ and $g:Y \to Z$ which is associative.
A \emph{subcategory} $S$ of $C$ is a subset of objects and morphisms which is itself a category. 
A subcategory $S \subseteq C$ is \emph{full} if $S(X,Y) = C(X,Y)$ for any objects $X,Y \in S$.

For example, we write $\Vec_\mathbb{F}$ for the category of finite-dimensional vector spaces over a fixed field $\F$  with the morphisms given by linear transformations between vector spaces. 
Another common example is that of the category of sets with set maps, denoted $\Set$. 

We can also obtain a category from any poset $P$. 
This category has an object for every element $p\in P$ and a single morphism $p \to q$ for any pair $p \leq q$ in $P$. 
This category is a \emph{thin} category which means there is at most one morphism for any pair of objects. %

A \emph{functor} $F:C \to D$ between categories $C$ and $D$ is a map which takes objects $X$ of $C$ to objects $F(X)$ in $D$ and morphisms $f$ in $C(X,Y)$ to morphisms $F[f]$ in $D(F(X), F(Y))$. 
These mappings are compatible with the identity morphism and the associative property of the composition.

A special type of functor is a \emph{diagram} on $C$, a functor to $C$ from a category $J$ with a (very) small set of objects and a (very) small set of morphisms.  
The category $J$ is called the index category of the diagram. 
Visually, a diagram is a directed graph with  objects of $C$ as vertices  morphisms as edges. 
See Fig.~\ref{fig_diagram}.
We say that the diagram commutes if any two paths of composable arrows in the directed graph with common source and target have the same composition.
For instance, the diagram of Fig.~\ref{fig_diagram} commutes if $h \circ f = i \circ g$.

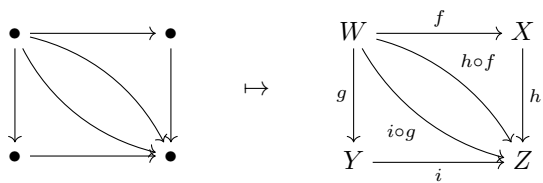
\begin{figure}[h]
    \centering
    \[\begin{tikzcd}[cramped]
	{\bullet} && {\bullet} \\
	\\
	{\bullet} && {\bullet}
	\arrow[" ", from=1-1, to=1-3]
	\arrow[" "', from=1-1, to=3-1]
	\arrow[" ", from=1-3, to=3-3]
	\arrow[" "', from=3-1, to=3-3]
    \arrow[" ", curve={height=-12pt},  from=1-1, to=3-3]
    \arrow[" ", curve={height=12pt},  from=1-1, to=3-3]   
\end{tikzcd}
\qquad \mapsto \qquad 
    \begin{tikzcd}[cramped]
	{W} && {X} \\
	\\
	{Y} && {Z}
	\arrow["f", from=1-1, to=1-3]
	\arrow["g"', from=1-1, to=3-1]
	\arrow["h", from=1-3, to=3-3]
	\arrow["i"', from=3-1, to=3-3]
    \arrow[" h \circ f ", curve={height=-12pt},  from=1-1, to=3-3]
    \arrow[" i \circ g"', curve={height=12pt},  from=1-1, to=3-3]
\end{tikzcd}\]

    \caption{A index category $J$ and square diagram on a category $C$ with objects $W,X,Y,Z$ and morphisms $f,g,h,i$ }
    \label{fig_diagram}
\end{figure}

In particular, a finite poset category is a small category. Thus for for a finite poset $P$ and a category $C$, a functor $F : P \to C$ is a diagram of $C$. Let us add a useful, immediate lemma that applies for the poset categories.

 \begin{lemma}
     [Thin lemma] In a thin category all diagrams commute given there is at most one morphism between any two objects.
     \label{lemma_thin}
 \end{lemma}

As a consequence of Lem.~\ref{lemma_thin} and the definition of functors we obtain the following corollary.

\begin{corollary}
\label{coro_split}
    For a finite poset $P$, a category $C$ ,the diagram  $F : P \to C$ and  $p \leq r \leq q$ in $P$ we have that
\[
    F[p \leq q] = F[(r \leq q) \circ (p \leq r)]= F[r \leq q] \circ F[p \leq r].
\]
\end{corollary}

There are many special cases of morphisms to consider.
A morphism $f:X \to Y$ is called a \textit{monomorphism} if for all objects $Z\in C$ and morphisms $g,h: Z \to X$ such that $f \circ h = f \circ g$,  then $g =h$. 
Similarly, $f$ is called an \emph{epimorphism} if for all objects $Z$ and morphisms $h,g: Y \to Z$ such that $ h \circ  g = g \circ f$, then $g =h$. 
Finally, $f$ is an \emph{isomorphism} if there is $g: Y \to X$ such that $f \circ g = id_Y $ and $g \circ f = id_X $. 
Note that if $f$ is an isomorphism then it is monomorphism and epimorphism as well.
For example, in $\Vec_\F$ the monomorphisms coincide with injective maps and epimorphism coincide with surjective maps. 
A map is bijective if it is injective and surjective. 
In this category, bijective maps coincide with the isomorphisms.

Similarly to morphisms, we have different types of functors. 
Let $C$ and $D$ be locally small categories. A functor $F: C \to \D$ is \emph{faithful} if the mapping $C(X,Y) \to D(F(X), F(Y))$ is injective and is \emph{full} if such mapping is surjective.
We say $C$ is a a \emph{concrete category} if it is equipped with a faithful functor $U: C \to \Set$, also called a forgetful functor.
If $C$ is a concrete category we can think of its objects as sets with additional structure and of its morphisms as structure-preserving functions. 
We will call a concrete category whose underlying sets are finite a \emph{FC category}.

Take $F,G: C \to D$ to be functors. A \emph{natural transformation} $\phi: F \Rightarrow G$ consists of a family of morphisms indexed by the objects of $C$, $\{\phi_X: F(X) \to G(X) \mid X \in C\}$, such that for $f: X \to Y$ the diagram 
\[\begin{tikzcd}[cramped]
	{F(X)} && {G(X)} \\
	\\
	{F(Y)} && {G(Y)}
	\arrow["{\phi_X}", from=1-1, to=1-3]
	\arrow["{F[f]}"', from=1-1, to=3-1]
	\arrow["{G[f]}", from=1-3, to=3-3]
	\arrow["{\phi_Y}"', from=3-1, to=3-3]
\end{tikzcd}\]
commutes.
If $\phi_X$ is an isomorphism for all objects $X$,  we say that $\phi$ is a natural isomorphism. 
For fixed $C$ and $D$, there is a \emph{functor category}, written $D^C$, with objects given by functors and morphisms given by natural transformations. 
We write $\End(C) = C^C$ for the \emph{endomorphisms} of a category $C$. 

We have the following corollary of the Thin Lemma (Lem.~\ref{lemma_thin}).

\begin{corollary}
     Fix a poset $P$, any category $C$ and functors $F,G: C \to P$.
     There is a natural transformation $\phi: F \Rightarrow G $  if and only if $F(X) \leq G(X)$ for all $X \in C$.
We write $ F \leq G $ if such a $\phi$ exists.
\end{corollary}

Take a finite poset $P$ and let $P^{\text{op}}$ denote its opposite category. The objects of  $P^{\text{op}}$ are the elements $p\in P$ and for  $p \leq q$ in $P$  there is a single morphism $q \to p$. Consider a diagram $F: P^{\text{op}} \to C$ on some category $C$. 

    Take $N$ and object of $C$ and define the functor $\Delta_N: P^{\text{op}} \to C$ given by $\Delta_N(p) = N$ for all $p \in P$. A \textit{cone} from $N$ to $F$ is a natural transformation $\phi: \Delta_N \Rightarrow F$. We say that $N$ is the \textit{vertex} of the cone, as shown in the following diagram:
    \[\begin{tikzcd}[cramped]
        & N &\\
        & L  &\\ \\
        {F(q)} && {F(p)}.
        \arrow["u"', dashed, from=1-2, to=2-2]
        \arrow["{\phi_q}"', curve={height=3pt}, from=1-2, to=4-1]
        \arrow["{\phi_p}", curve={height=-3pt},from=1-2, to=4-3]
        \arrow["{\psi_q}", from=2-2, to=4-1]
        \arrow["{\psi_p}"', from=2-2, to=4-3]
        \arrow[from=4-1, to=4-3]
    \end{tikzcd}\]
    
   Let $L \in C$. We say that $L$ is the \textit{inverse limit} of $F$ if there is \mbox{$\psi: \Delta_L\Rightarrow F$} a cone from $L$ to $F$ and for any other cone $\phi: \Delta_N \Rightarrow F$ there is a unique morphism $u$ such that for all $p \in P$, $\psi_p \circ u = \phi_p $. We denote it by $L = \displaystyle \lim_{\longleftarrow} (F(p) \mid p \in P)$. 
    The defining property of the inverse limit is called a \textit{universal property} because any other cones factor through the inverse limit; i.e.~$\psi \circ u = \phi$.

\subsection{Flows on a poset }
\label{ssec:FlowsOnPosets}

Since our goal is to define the interleaving distance for functor categories $C^P$, we now discuss how to define a flow on a poset, following \cite{Bubenik2015} and \cite{Flow}.
A \emph{translation} on $P$,  $\mathcal{S} \in \End(P)$,  is an endomorphism of poset $P$ such that $\Id \leq \mathcal{S}$  where  $\Id:P \to P$ is the identity endomorphism.
Let $\Trans(P)$ denote the set of translations of $P$, where we view $\Trans(P)$ as a full subcategory of $\End(P)$. 
\begin{definition}
\label{defn:flow}
A \emph{flow} $\T$ on a poset $P$ consists of a functor 
$\T: ([0,\infty), \leq) \to \Trans(P)$ 
with $\T(\e) = \T_\e$ such that 
$\T_\e \circ \T_{\e'}  \leq \T_{\e + \e'}$. 
We say that $\T$ is \emph{commutative} if additionally 
$ \T_{\e'} \circ \T_{\e} =  \T_\e \circ \T_{\e'}$. 
We call $\T$ a \emph{strict flow} if $\T_0 = \Id $  and 
$\T_{\e} \circ \T_{\e'} =  \T_{\e + \e'} $. 
Explicitly, a \emph{flow} is a function $\T: ([0,\infty), \leq) \to \End(P)$ satisfying four conditions:

    \begin{enumerate}[(1)]
        \item For all $0 \leq \e $ and for all $p \leq q \in P$, $\T_\e p \leq \T_\e q $.
        \item For all $0 \leq \delta < \e $ and for all $p \in P$, $\T_\delta p \leq \T_\e p $.
        \item For all $0 \leq \e $ and for all $p \in P$, $p \leq \T_\e p $.
        \item For all $0 \leq \e ,\e'$ and for all $p \in P$, \mbox{$\T_\e \circ \T_{\e'} p \leq \T_{\e + \e'} p $} .
    \end{enumerate}
\end{definition}

We are going to be particularly interested in the translations attained by finitely many compositions of morphisms from some input $\T$, defined as follows. 

\begin{definition}
\label{def:flowsCategory}
The full subcategory $\Flows_{\T} \subseteq \Trans(P) \subseteq \End(P)$ is given by the translations of $P$ that are finite compositions of morphisms $\T_\e$,
$$
\Flows_{\T} = \{ \T_{\e_1} \circ \cdots \circ \T_{\e_\ell} \mid \ell \geq 0, \e_i \geq 0 \}.
$$
\end{definition}
If the flow  is commutative (i.e.~$ \T_{\e'} \circ \T_{\e} =  \T_\e \circ \T_{\e'}$), the composition in $\Flows_{\T}$ is commutative (i.e., for $\mathcal{S}, \mathcal{S}' \in \Flows_{\T}$ we have $\mathcal{S} \circ \mathcal{S}' = \mathcal{S}' \circ \mathcal{S}$).

\begin{definition}
\label{def:Archimedean}
We say that a flow $\T$ on poset $P$ is an \emph{Archimedean flow} if for all $p < q$ in $P$ there is $\delta$ such that $\T_\delta \; p \geq q $.
\end{definition}
For example, define the flow 
$\T: [0,\infty) \to \Trans(\R)$ by $T_\e(x) =  x + \e$. 
For $x < y$, we can check that $\T_{y-x}(x) = y$, so $\T$ is Archimedean.
We have the following result in the special case that $\T$ is a flow on $\R$.

\begin{lemma}
    A flow  $\T$ on $\R$  is Archimedean if and only if for all $x \in \R$ there is $M_x >0$ such that if $\e \geq M_x$, then   $\T_\e(x) > x$ .
\end{lemma}
\begin{proof}
    The "only if" statement is proved by contrapositive.
    Take a flow with an $x$ such that for all $M > 0$ and $\e \leq M$, $\T_\e(x) \leq x$. Then, for all $M > 0$, $ x \leq \T_M(x) \leq x$. Such flow fails to be Archimedean given that for $ y > x $ there is no $\e$ with $\T_\e(x) \geq y$.
    
    For the "if" statement take $\T$ to be a flow on $\R$ with the condition  that for all $x \in \R$ there is $M_x >0$ such that if $\e \geq M_x$, then   $\T_\e(x) > x$ . 
    We can write $\T$ as a function $\T: [0,\infty) \times \R \to \R$. 
    Using Condition 3), define $\phi:\R_{\geq 0} \times \R \to \R_{\geq 0}$ by $\phi(\e,x) = \T_\e(x)-x$.
    By Conditions 1) and 2), $\phi$ is non-decreasing in both entries. 
    Further, because of our assumptions on $\T$, $\phi(\e,x) >0$ for $\e \geq M_x$
    Then, by condition 4)
    \begin{align*}
        x + \phi(\e',x) + \phi(\e,x) 
        &\leq x + \phi(\e',x) + \phi(\e,x + \phi(\e',x)) \qquad  \text{(Non decr in 2nd entry)} \\
        & = \T_\e \circ \T_{\e'} (x) \\
        &\leq\T_{\e+\e'} (x)  \qquad \text{(Condition (4))}\\
        & = x + \phi(\e + \e' ,x).
    \end{align*}
    Thus, 
    $\phi(\cdot,x)$ is superadditive, meaning $\phi(\e',x) + \phi(\e,x) \leq \phi(\e + \e' ,x)$. 
    We prove by induction that for $n \in \N$, $n \phi(\e,x) \leq \phi(n \e,x) $.

    By way of contradiction, suppose that $\T$ is not Archimedean so there exists $x < y$ in $\R$ such that $\T_\e x < y$ for all $\e > 0$. 
        Hence, for a fixed but arbitrary $\e \geq M_x$,
     $$y  > \T_{n\e}(x) 
     = x + \phi(n\e,x) 
     \geq x + n\phi(\e,x). $$
     Dividing by $n$ and taking the limit when $n$ goes to infinity we obtain that $0 \geq \phi(\e,x)$. 
    However, this contradicts that $\phi(\e,x)>0$, finishing the proof. 
\end{proof}

\begin{definition}
    Let $T$ be linear poset. 
    A translation $\mathcal{S}$ on $T^k$ is called \emph{line-preserving} if $\mathcal{S}$ satisfies if it satisfies either of the equivalent conditions of the subsequent lemma.
\end{definition}

\begin{lemma}
    \label{lemma:line-pres}
     Let $T$ be linear poset and $\mathcal{S}$ a translation on $T^k$. The following are equivalent:
     \begin{enumerate}[i)]
         \item For each $ 1 \leq i \leq k$ there are translations $\mathcal{S}^i: T \to T$ such that for $\mathbf{p}=\ncoord{p}{k} \in T^k $ we have that $\mathcal{S}(\mathbf{p}) =  \anglebracket{\mathcal{S}^1(p^1), ..., \mathcal{S}^k(p^k)}$.
         \item For any $i$ between $1$ and $k$, if $\mathbf{p}, \mathbf{q}$ in $T^k$ are such that all their coordinates are the same except for $i$-th coordinate, then the coordinates of $\mathcal{S}(\mathbf{p})$ and $\mathcal{S}(\mathbf{q})$ are the same except, maybe, for the $i$-th coordinate.
     \end{enumerate}
\end{lemma}
\begin{proof}
The statement i) $\implies$ ii) is straightforward. 

For the reverse implication assume ii). Take a fixed $j$ between $1$ and $k$. Let us prove that for $ \mathbf{p},  \mathbf{q}$ in $T^k$ such that $p^j =q^j$ we have that $(\mathcal{S}(\mathbf{p}))^j =  (\mathcal{S}(\mathbf{q}))^j$.
Then we can define $S^j(t) $ to be this common values for any $\mathbf{p} \in T^k$ with $p^j = t$. 

Take different $ \mathbf{p},  \mathbf{q}$ satisfying that condition in $T^k$. Consider the non-empty set $U = \{ i \mid p^i = q^i\}$. Let us proceed by reverse induction. If $|U| = k-1$, by assumption we obtain then  $\mathcal{S}(\mathbf{p})$ and $\mathcal{S}(\mathbf{q})$ are the same except, maybe, for the coordinate not in $U$. In particular,  $(\mathcal{S}(\mathbf{p}))^j =  (\mathcal{S}(\mathbf{q}))^j$. Assume that if $ k-1 \geq |U| \geq \ell$ it implies that $(\mathcal{S}(\mathbf{p}))^j =  (\mathcal{S}(\mathbf{q}))^j$. Suppose that $|U| = \ell - 1$. For fixed $j',j'' \notin U $ choose $\mathbf{r}$ with $r^i = p^i$ if $i \in U$,  $r^{j'} = p^{j'}$ and $r^{j''} = q^{j''}$. Then, the sets $\{ i \mid p^i = r^i\}$ and $\{ i \mid q^i = r^i\}$ have a cardinality strictly greater than $\ell - 1$, hence $$(\mathcal{S}(\mathbf{p}))^j =  (\mathcal{S}(\mathbf{r}))^j = (\mathcal{S}(\mathbf{q}))^j.$$
\end{proof}

Consider the example map $\T: [0,\infty) \to \Trans(\N^2)$ given by 
\begin{equation}
\label{eq:ex-not_line_preserving}
T_\e(\anglebracket{m,n}) =  \anglebracket{m,n+\floor{\e} m}. 
\end{equation}
Let us check that $\T$ is a flow checking each condition from Defn.~\ref{defn:flow}.
\begin{enumerate}[1)]
    \item Take $0 \leq \e$ and $\anglebracket{m,n} \leq \anglebracket{m' ,n' }  $. Then $\anglebracket{m,n+\floor{\e} m} \leq \anglebracket{m',n'+\floor{\e} m'} $.
    \item Take $0 \leq \delta < \e$ and any $\anglebracket{m,n} $. Then $\anglebracket{m,n+\floor{\delta} m} \leq \anglebracket{m,n+\floor{\e} m} $.
    \item Take $0 < \e$ and any $\anglebracket{m,n} $. Then $\anglebracket{m,n} \leq \anglebracket{m,n+\floor{\e} m} $.
    \item Take $0 < \e, \e' $ and any $\anglebracket{m,n} $. Then $\anglebracket{m,n + \floor{\e}m + \floor{\e'}m} \leq \anglebracket{m,n+\floor{\e + \e'} m} $.
\end{enumerate}
By Lemma \ref{lemma:line-pres}, the translation $\T_\e$ of Eqn.~\eqref{eq:ex-not_line_preserving} for $\e \geq 1$ is not line-preserving given that $n + \floor{\e}m$ is not independent of $m$ in that case.
    
Note that a line preserving flow is not necessarily injective. 
For example, consider $\T: [0,\infty) \to \Trans(\R^2)$ given by $T_0(\anglebracket{m,n}) = \anglebracket{m,n}$ and $T_\e(\anglebracket{m,n}) =  \anglebracket{m+ \e,\ceil{n} + \floor{\e}}$ for $\e >0$.  This is line preserving because the function $m + \epsilon$ is independent of $n$ and the function $\ceil{n} + \floor{\e}$ is independent of $m$, but is not injective because $T_\e(\anglebracket{1,1.5}) = T_\e(\anglebracket{1,2})$.

\subsection{Generalized persistence modules}
\label{ssec:genPersMod}

For a poset $P$ and a category $C$ , functors $F: P \to C$ are called \emph{generalized persistence modules} in $C$ over $P$ \citep{Bubenik2015}. 
The collection of these functors is itself a (thin) category with natural transformations as morphisms, denoted $C^P$. 

The terminology arises from the following standard construction in persistence. 
Let $ \emptyset =K_0 \subseteq K_1 \subseteq \cdots K_n = K$ be a filtration of a simplicial complex $K$. 
Denote by $H_k(-)$ the $k$th homology with coefficients in some field $\F$. Define the functor: $F: \{0, \ldots, n\} \to \Vec_\F$ given by $F(i) := H_k(K_i)$ and $F[i \leq j] := H_k(K_i) \to H_k(K_j)$ the map induced by $K_i \hookrightarrow K_j$. 
Then, $F$ is a generalized persistence module in  $\Vec_\F$ over $ \{0, \ldots, n\} $.
A module $F: P \to \Vec_\F$ is called a \emph{$q$-tame module} if for all $r < p \in P$ the corresponding maps $F[r \leq p]$  have finite rank.\footnote{
We note that the term $q$-tame is common in the literature (e.g.~\cite[Def 1.12]{Oudot_pers}), it is not related to any fixed choice of $q \in P$.}

For another example, consider the same filtration from before. 
Define the functor $G:  \{0, \ldots, n\}   \to \Set$ given by $G(i) := \pi_0(K_i)$ the set of connected components of $K_i$ with morphisms $G(i) \to G(j)$ induced by inclusion. 
Such module is called a \emph{merge tree} since it is the categorical representation of the eponymous topological construction that tracks the connectivity of $K$ over the filtration.

\subsubsection{Critical values of a persistence module}
\label{ssec:criticalVals}

 \begin{definition}
   Let $F\in C^P$ be a generalized persistence module and let $I \subseteq P$ be an interval.
   We say that $F$ is \emph{constant on} $I$, or that $I$ is a constant interval of $F$ if for all $p \leq q \in I$ , $F[p \leq q]$ is an isomorphism. 
\end{definition}

\begin{definition}
    Let $T$ be a linear poset and fix $t \in T$. A subset $I \subseteq T$ is called an \emph{open interval around} $t$ if:
    \begin{itemize}
        \item The set $I$ is equal to $(x,y)$, $(x,\infty)$ or $(-\infty,y)$  for some $x < t < y \in T$.
        \item Whenever $t$ has an upper bound in $T$, there exists $u \in I$ with $t < u$.
        \item Whenever $t$ has a lower bound in $T$, there exists $s \in I$ with $s < t$.
    \end{itemize}
\end{definition}

Thus, $I$ is an open interval in the usual sense and includes $t$, one element above $t$ (when one exists) and one element below $t$ (when one exists).
For example, in the case where  $T = \Z$, the set $I = \{2\} \subset \Z$ is of the form $(1,3)$, however, it is not an open interval around $t=2$ because there is no $u \in I$ for which $2<u$. 
The set $(0,4)$ however would be a open interval around $t=2$.

\begin{definition}
\label{def:tameModule}
    Let $T$ be a complete linear poset and fix a $F \in C^T$. 
    We say that $t \in T$ is a \emph{critical value} of $F$ if for any open interval around $t$, $F$ is not constant on $I$. 
    Otherwise, we say that $t$ is a \emph{regular value} of $F$. 
    We call $F \in C^T$ \emph{tame} if it has a finite number of critical values. 
    In that case, the set of all critical values of $F$, $ \{\tuple{c}{m}\}$, is called the \emph{critical set} of $F$. 
\end{definition}

The notions of critical values and tame modules in categorical terms were introduced by Bubenik and Scott in \cite{categorification} for $P = \R$. 
The definitions presented here apply to any complete linear poset. 
We can consider $\R$ to be complete adding $-\infty$ and $\infty$ as the minimal, respectively maximal of the order.

\begin{lemma}[Critical Value Lemma] 
\label{lemma_critical}
    An interval $I \subseteq T$ does not contain any critical values of $F: T \to C$, if and only if $F$ is constant on $I$.
\end{lemma}
\begin{proof}
    Let $I$ be an interval that does not contain a critical point. The sufficient condition follows from the Critical Value Lemma (4.4) of \cite[Lem.~4.4]{categorification} join with the Lemma \ref{lemma_compact} to satisfy the proper conditions. 
    The necessary condition comes by the fact that an interval containing a critical point of $F$ can not be a constant interval of $F$.
\end{proof}

\begin{definition}
    Fix a complete linear poset $T$ and a generalized persistence module $F: T^k \to C$. 
    We say that $t \in T$ is a \emph{critical $j$-coordinate} of $F$ if there are some $t_1, \ldots , t_{j-1},t_{j+1}, \ldots, t_k \in T$ such that $(t_1, \ldots , t_{j-1},t,t_{j+1}, \ldots, t_k)$ is a critical value of $F_j:T \to C; s \mapsto F((t_1, \ldots , t_{j-1}, s , t_{j+1}, \ldots t_k ))$
\end{definition}

\begin{definition}
\label{def:cubicalTame}
    We call $F: T^k \to C$ \emph{cubical tame} if it has a finite number of critical coordinates.
\end{definition}

\subsubsection{Distance between persistence modules}
\label{ssec:interleaving}

We next define a distance in $C^P$ using the notion of flows from Sec.~\ref{ssec:FlowsOnPosets}.

\begin{definition}
Assume the poset $P$ is endowed with a commutative flow $\T$. 
We say that $F,G \in C^P$ are $\T_\e$-interleaved if we have a pair of natural transformations 
$$\Phi: F \Rightarrow G\T_{\e} \qquad \Psi: G \Rightarrow F\T_{\e} $$
such that the diagrams 
\begin{equation*}
    \begin{tikzcd}[cramped]
    	{F(p)} && {F(\mathcal{T}_\e\mathcal{T}_\e p)} &&& {F(\mathcal{T}_\e p)} \\
    	& {G(\mathcal{T}_\e p)} &&& {G(p)} && {G(\mathcal{T}_\e\mathcal{T}_\e p)} 
    	\arrow["{F[p \leq \mathcal{T}_\e \mathcal{T}_\e p]}", from=1-1, to=1-3]
    	\arrow["{\Phi_p}"', from=1-1, to=2-2]
    	\arrow["{\Phi_{\mathcal{T}_\e p}}", from=1-6, to=2-7]
    	\arrow["{\Psi_{\mathcal{T}_\e p}}"', from=2-2, to=1-3]
    	\arrow["{\Psi_p}", from=2-5, to=1-6]
    	\arrow["{G[p \leq \mathcal{T}_\e \mathcal{T}_\e p]}"', from=2-5, to=2-7]
    \end{tikzcd}
\end{equation*}
commute.
The interleaving distance is defined to be 
$$  
d_I(F,G) = \inf \{\e \geq 0 \mid F,G \text{ are } \T_\e \text{-interleaved}\}
$$
and we set $d_I(F,G) = \infty$ if there is no $\T_\e$-interleaving.
\end{definition}

\begin{proposition}[{\cite[Thm.~3.15]{Bubenik2015}}]
   The interleaving distance $d_I$ is an extended-pseudometric on $C^P$.
\end{proposition}

\section{General loss function bounding the interleaving distance}

\label{sec_loss}

In this section, we give a bound for the interleaving distance for generalized persistence modules by defining a loss function inspired by \cite{2024bounding}. 
Throughout this section, we assume that $F$ and $G$ are generalized persistence modules in $C^P$ for  a concrete category $C$ and a poset $P$ with a commutative flow $\T$.

\subsection{The merging distance function}
\label{ssec:MergingDistance}

Fix $p \in P$. 
We will build a notion of distance on $F(p)$ by measuring how much "we need to flow" on $P$ to attain an element equality. 
For that reason, we will assume that $\T_0 = \Id$.

\begin{definition}
    \label{def_merging}
    For $a,b \in F(p)$, define
    \begin{equation*}
        d_p^F(a,b) = \inf\{\e \geq 0 \mid F[p \leq T_{\e} p](a) =  F[p \leq T_{\e} p](b) \}.
    \end{equation*}
    Since the function is indeed a distance (which will be shown in Theorem \ref{th_ultrametric}), we call it the \textit{merging distance} function. 
\end{definition}

\begin{lemma}
    Let $a,b \in F(p)$. 
    If there is $\cS \in \Trans(P)$ such that 
    $F[p \leq \cS\,p](a) =  F[p \leq \cS\,p](b)$, 
    then for any $\cS \leq \cS' $, 
    $F[p \leq \cS'\, p](a) =  F[p \leq \cS' \, p](b)$.
    \label{lemma_slide}
\end{lemma}
This can be restated as follows: if the translation by $\cS$ in $F$ attains the equality of $a$ and $b$, then the equality is attained by any translation  $\cS' \geq \cS$.

\begin{proof}[Proof of Lem.~\ref{lemma_slide}]
We compute
    $$
    F[p \leq \cS' p](a) = F[\cS\,p \leq\cS' p] F[p \leq \cS p] (a) =  F[\cS \, p \leq \cS'\, p] F[p \leq \cS\, p] (b) = F[p \leq \cS'\, p](b) 
    $$
where the first and third equality follow from Cor.~\ref{coro_split} and the second equality follows from the assumption.
\end{proof}

Then we have the following immediate corollary.

\begin{corollary}
    For any $\delta > d_p^F(a,b)$ we have that $F[p \leq \T_{\delta} p](a) =  F[p \leq T_{\delta} p](b)$.
    \label{coro_slide}
\end{corollary}

Next, we check the distance properties of $d_p^F$. 
\begin{theorem}
    \label{th_ultrametric}
    The merging distance $d_p^F$ is an extended pseudo-ultrametric on $F(p)$. 
\end{theorem}

\begin{proof}
    For any $a,b \in P$, it is immediate from the definition that  $d_p^F(a,b) = d_p^F(b,a)$,  that $ d_p^F(a,b) \geq 0$, and that $d_p^F(a,a) = 0$.  
    However it can be the case that $d_p^F(a,b) = 0$ but $a$ and $b$ are not equal, see Example \ref{example_1}.
    The distance can take infinite values since $d_p^F(a,b) = \infty$ if there is no $\e \geq  0$ such that $F[p \leq T_{\e} p](a) =  F[p \leq \T_{\e} p](b)$.
    
    Finally we show that 
    $$  d_p^F(a,b) \leq \max\{d_p^F(a,c), d_p^F(c,b)\}$$ 
    for any $c \in F(p)$ for the ultrametric property.
    If either $d_p^F(a,c)$ or $d_p^F(c,b)$ are infinite we are done, so without loss of generality, assume $d_p^F(a,c) \leq d_p^F(c,b) < \infty $. 
    Choose $0 \leq \e \leq \e'$ such that 
    $$
    F[p \leq \T_{\e} p](a) =  F[p \leq \T_{\e} p](c)
    \qquad and \qquad
    F[p \leq \T_{\e'} p](c) =  F[p \leq \T_{\e'} p](b).
    $$
    Then, by Lemma \ref{lemma_slide}, $ F[p \leq \T_{\e'} p](a) = F[p \leq \T_{\e'} p](c) = F[p \leq \T_{\e'} p](b)
    $. Hence, since $\e'$ is arbitrary we obtain that $d_p^F(a,b) \leq d_p^F(c,b)$. 
\end{proof}

\begin{example}
     \label{example_1}  
     Define the flow $\T: [0,\infty) \to \Trans([0,\infty))$ by $T_\e(m) =  m + \e$. Consider the modules $F$ and $F'$ in $\Set^{[0,\infty)}$ given by
    \begin{equation*}
         F(i) = \left\{ \begin{array}{cc}
             \{a,b\} & i <3 \\
              \{a\} &  i \geq 3
         \end{array}\right. \hspace{3cm}
         F'(i) = \left\{ \begin{array}{cc}
             \{a,b\} & i\leq 3 \\
              \{a\} &  i > 3
         \end{array}\right.     
     \end{equation*}
     with morphisms sending $a \to a$ and $b \to b$ when possible, and $b \to a$ when its not. 
     Visually, we can think of $F$ as the diagram of maps 
   \[
   \begin{tikzcd}[cramped, row sep = tiny]
   F(0) \ar[r]
   & F(1) \ar[r] 
   & F(2) \ar[r]
   & F(3) \ar[r] 
   & \cdots \ar[r]
   & F(i) \ar[r]
   & \cdots 
   \\
     {a_0} \ar[r,mapsto]
    & {a_1} \ar[r,mapsto]
    & {a_2} \ar[r,mapsto]
    & {a_3} \ar[r,mapsto]
    & {\cdots} \ar[r,mapsto]
    & {a_i} \ar[r] 
    & \cdots 
    \\
	 {b_0} \ar[r,mapsto]
    & {b_1} \ar[r,mapsto]
    & {b_2} \ar[ur,mapsto]
\end{tikzcd}\]

and $F'$ is the same except for the map into $F(3)$. 
Note that $F$ and $F'$ are almost the same, they have the same critical point but  change the size of sets differently on either side of 3. 
We can check that $d_i^F(a,b) = d_i^{F'}(a,b) = 3 - i$ for $i < 3$ and $d_3^{F'}(a,b) = 0$. This shows we may have different elements with distance 0. 

Also, notice that this distance corresponds to "how far" in terms of the flow is $i$ is from $3$. This gives us that for $i < j < 3$ and $\delta \geq (j-i)$
\[
d_j^F(a,b) < d_i^F(a,b) \leq d_j^F(a,b) + \delta
\]
\end{example}

This result is formalized in the following lemmas.

\begin{lemma}
    \label{lemma_preimage}
    If $p \leq q$ in $P$ and $b,b' \in F(p)$ with $c = F[p \leq q](b)$ and $c' = F[p \leq q](b')$, 
    then $d_{q}^F (c,c') \leq  d_{p}^F (b,b')$.
\end{lemma}

\begin{proof}
Fix $\delta > d_{p}^F (b',b)$. 
By Corollary \ref{coro_slide} we have that $F[p \leq \T_\delta p](b) = F[p \leq \T_\delta p](b') =: x$. 
 By the Thin Lemma, the diagram 
 \begin{equation*}
 \begin{tikzcd}
     F(p) \ar[r] \ar[d] & F(q) \ar[d] 
     & \substack{b\\b'} \ar[r, mapsto, shift left] \ar[r, mapsto, shift right] \ar[d, mapsto] 
     & \substack{c\\c'} \ar[d, dashed, mapsto] \\
     F(T_\delta p) \ar[r] & F(T_\delta q) 
     & x \ar[r, mapsto] & y
 \end{tikzcd}
 \end{equation*}
 commutes taking $y = F[T_\delta p \leq T_\delta q](x).$
So since $b,b'$ map to the same element, $c$ and $c'$ also map to the same element, explicitly, $y = F[q \leq \T_\delta q](c) = F[q \leq \T_\delta q](c')$.
This implies that $\delta \geq d_{q}^F (c',c)$ and, since the $\delta$ is arbitrary, by definition of infimum the conclusion follows.
\end{proof}

Intuitively, Lemma \ref{lemma_preimage} says that the merge distance is non-increasing when flowing forward. 
Further, by setting $q = \T_\delta \; p$ in Lemma ~\ref{lemma_preimage} for any $\delta \in (d_{p}^F (b',b),\infty)$,  by Cor.~\ref{coro_slide} for any $c,c' \in F(T_\delta \; p) = F(q)$,  we find that $d_{q}^F (c',c) = 0$.

When it comes to bound $ d_{p}^F (b',b)$ in terms of $d_{q}^F (c,c')$ the Archimedean property of the flow is a sufficient condition 

\begin{lemma}
    \label{lemma_preimage_Archimedean}
    Let $\T$ be an Archimedean flow on $P$. If $p \leq q$ in $P$ and $b,b' \in F(p)$ with $c = F[p \leq q](b)$ and $c' = F[p \leq q](b')$, then, there is $\delta$ such that $d_{q}^F (b,b') <  d_{p}^F (c,c') + \delta$.
\end{lemma}
\begin{proof}
    Following the same notation as the previous lemma, we will show there is $\delta$ such that $d_{p}^F(b,b') < d_{q}^F(c,c') + \delta$. To find such a $\delta$, take 
    $\gamma > d_{q}^F(c,c')$. 
    Then, by Cor.~\ref{coro_slide} this gives that 
    $F[q \leq T_\gamma q](c) = F[q \leq T_\gamma q](c') $.
    Thus 
    \begin{align*}
    F[p \leq T_\gamma q](b) 
    &= F[q \leq T_\gamma q]F[p \leq q](b)\\
    &   =F[q \leq T_\gamma q](c) \\
    &   =F[q \leq T_\gamma q](c') \\
    &   = F[q \leq T_\gamma q]F[p \leq q](b')\\
    &= F[p \leq T_\gamma q](b'). 
    \end{align*}
    Because $\T$ is a flow, 
    $p < \T_\gamma q$, so by the Archimedean assumption, there is a $\hat{\delta}$ such that $\T_{\hat{\delta}} p > T_\gamma q$. 
    Note that along with $\T_{\hat{\delta}} q \geq \T_{\hat{\delta}} p > T_\gamma q$, this implies that $\hat{\delta} > \gamma$, so 
    \begin{equation*}
    F[p \leq \T_{\hat{\delta}} p](b) = F[p \leq \T_{\hat{\delta}} p](b')
    \end{equation*}
    Hence setting $\delta = \hat{\delta} -d_{q}^F(c,c') $, we have 
    \[
    \gamma+\delta = \gamma + (\hat{\delta} -d_{q}^F(c,c')) > \gamma + (\hat{\delta} - \gamma) = \hat{\delta} > d_{p}^F(b,b')
    \]
    Since $\gamma$ was arbitrary we obtain that $d_{q}^F(c,c') + \delta > d_{p}^F(b,b')$ finishing the proof.
\end{proof}

\subsection{The loss function}
\label{ssec:lossfunction}
Our next job is to define structures that take the form of natural transformations, but lack the usual guarantees. 

\begin{definition}
    Given functors $F,F' \in C^P$, an \textit{unnatural transformation} 
    $\eta: F \rightdsar F'$ 
    is a collection of maps $\eta_p: F(p) \to F'(p)$ with no additional promise of commutativity. 
    For a fixed $\cS \in \Flows_{\T}$ (Def.~\ref{def:flowsCategory}) and functors $F,G \in C^P$, an $\cS$-\textit{assignment of} $F$ and $G$ is a pair of unnatural transformations $\phi: F \rightdsar G\cS$ and $\psi: G \rightdsar F\cS$ which
    we write as $(\phi,\psi)$.
\end{definition}

For fixed $p,q \in P$ and $(\phi,\psi)$ a $\cS$-assignment of $F$ and $G$ consider the following diagrams indexed by finite categories in the category $C$. 
\[
\begin{tikzcd}[cramped, row sep =10pt]
	& \bullet && \bullet &&& {F(p)} && {F(q)} \\
	{\Parallelograml_\phi(p,q):} &&&& {} && {} \\
	&& \bullet && \bullet &&& {G(\cS p)} && {G(\cS q)}
	\arrow[from=1-2, to=1-4]
	\arrow[from=1-2, to=3-3]
	\arrow[curve={height=-6pt}, from=1-2, to=3-5]
	\arrow[curve={height=6pt}, from=1-2, to=3-5]
	\arrow[from=1-4, to=3-5]
	\arrow["{F[\; \leq \;]}"{description}, from=1-7, to=1-9]
	\arrow["{\phi_p}"{description}, from=1-7, to=3-8]
	\arrow["{G[\;\leq \;] \circ \phi_p}"{description}, curve={height=12pt}, from=1-7, to=3-10]
	\arrow["{\phi_q \circ F[\; \leq \;]}"{description}, curve={height=-12pt}, from=1-7, to=3-10]
	\arrow["{\phi_q}"{description}, from=1-9, to=3-10]
	\arrow[maps to, from=2-5, to=2-7]
	\arrow[from=3-3, to=3-5]
	\arrow["{G[\;\leq \;]}"{description}, from=3-8, to=3-10]
\end{tikzcd}
\]

\[\begin{tikzcd}[cramped, row sep =10pt]
	&& \bullet && \bullet &&& {F(\cS p)} && {F(\cS q)} \\
	{\Parallelogramr_\phi(p,q):} &&&& {} && {} \\
	& \bullet && \bullet &&& {G(p)} && {G(q)}
	\arrow[from=1-3, to=1-5]
	\arrow["{F[\;\leq \;]}"{description}, from=1-8, to=1-10]
	\arrow[maps to, from=2-5, to=2-7]
	\arrow[from=3-2, to=1-3]
	\arrow[curve={height=-6pt}, from=3-2, to=1-5]
	\arrow[curve={height=6pt}, from=3-2, to=1-5]
	\arrow[from=3-2, to=3-4]
	\arrow[from=3-4, to=1-5]
	\arrow["{\psi_p}"{description}, from=3-7, to=1-8]
	\arrow["{\psi_q \circ G[\;\leq \;] }"{description}, curve={height=6pt}, from=3-7, to=1-10]
	\arrow["{F[\; \leq \;] \circ \psi_p}"{description}, curve={height=-12pt}, from=3-7, to=1-10]
	\arrow["{G[\; \leq \;]}"{description}, from=3-7, to=3-9]
	\arrow["{\psi_q}"{description}, from=3-9, to=1-10]
\end{tikzcd}\]

\[\begin{tikzcd}[cramped, row sep =10pt]
	& \bullet && \bullet && {F(p)} && {F(\cS \cS p)} \\
	{\triangled_{\phi,\psi}(p):} &&& {} && {} \\
	&& \bullet &&&& {G(\cS p)}
	\arrow[from=1-2, to=1-4]
	\arrow[shift left, curve={height=18pt}, from=1-2, to=1-4]
	\arrow[from=1-2, to=3-3]
	\arrow["{F[\;\leq \;]}"{description}, from=1-6, to=1-8]
	\arrow["{\psi_{\cS p} \circ \phi_p}"{description}, shift left, curve={height=18pt}, from=1-6, to=1-8]
	\arrow["{\phi_p}"{description}, from=1-6, to=3-7]
	\arrow[maps to, from=2-4, to=2-6]
	\arrow[from=3-3, to=1-4]
	\arrow["{\psi_{\cS p}}"{description}, from=3-7, to=1-8]
\end{tikzcd}\]

\[\begin{tikzcd}[cramped, row sep =10pt]
	&& \bullet &&&& {G(\cS p)} \\
	{\triangleu_{\phi,\psi}(p):} &&& {} && {} \\
	& \bullet && \bullet && {G(p)} && {G(\cS \cS p)}
	\arrow[from=1-3, to=3-4]
	\arrow["{\phi_{\cS p}}"{description}, from=1-7, to=3-8]
	\arrow[maps to, from=2-4, to=2-6]
	\arrow[from=3-2, to=1-3]
	\arrow[from=3-2, to=3-4]
	\arrow[shift left, curve={height=-18pt}, from=3-2, to=3-4]
	\arrow["{\psi_p}"{description}, from=3-6, to=1-7]
	\arrow["{F[\;\leq \;]}"{description}, from=3-6, to=3-8]
	\arrow["{\phi_{\cS p} \circ \psi_p}"{description}, shift left, curve={height=-18pt}, from=3-6, to=3-8]
\end{tikzcd}\]

We give names for each diagram at the left side of each diagram. 
Our goal to measure the quality of a given assignment  $(\phi,\psi)$ by checking how far the parallel morphisms defined in each diagram are from being equal.
This is because if they were equal, the structure in question would indeed be a natural transformation (by the first two) and an interleaving (by the second two). 
This next definition writes the distance function between natural transformations using the distance already defined in the codomain, largely to simplify notation later. 

\begin{definition}
    Fix $q \in P$ and let $f,g: A \to F(q)$ be a pair of morphisms in $C$ with source $A \in C$. 
    We define the distance between $f$ and $g$ by
    \begin{equation*}
        d(f,g) = \sup_{a \in A} \left\{ d_{q}^F (f(a),g(a)) \right\}.
    \end{equation*}
\end{definition}

With this definition, we now have four measurements to be made for the diagrams above. 

\begin{definition}
\label{def:loss}
    Fix a $\cS$-assignment  $(\phi,\psi)$. 
    We define the loss of each diagram by 
    \begin{align*}
        \Lpl^{p,q}(\phi) &= d(\phi_q \circ  F[p \leq q], G[\cS p \leq \cS q] \circ \phi_p) & 
        \Ltd^p(\phi,\psi ) &= d(F[ p \leq \cS \cS p] , \psi_{\cS p} \circ \phi_p ) \\
        \Lpr^{p,q}(\psi) &= d( \psi_q \circ G[p \leq q] , F[\cS p \leq \cS q] \circ \psi_p) & 
       \Ltu^p(\phi,\psi ) &= d (G[ p \leq \cS \cS p] , \phi_{\cS p} \circ \psi_p).
    \end{align*}
    The loss for the given assignment is defined to be
    \begin{equation}
        \label{eq_L_general}
        L(\phi,\psi) = \sup_{p < q} \left\{  \Lpl^{p,q}(\phi),  \Ltd^q(\phi,\psi ), \Lpr^{p,q}(\psi) , \Ltu^q(\phi,\psi )\right\}.
    \end{equation}
\end{definition}

Note that a $\T_\e$-assignment $(\phi,\psi)$ with loss 0 is a $\T_\e$-interleaving between $F$ and $G$.   

\begin{example}
\label{example_F_G} 
As in Ex.~\ref{example_1}, define the flow $\T: [0,\infty) \to \Trans(\N)$ by $\T_\e(m) =  m + \floor{\e}$. 
Consider the merge trees $F$ and $G$ in $\Set^{\N}$ given by
   \[\begin{tikzcd}[cramped, row sep = tiny]
	{F:} 
    & {a_0} \ar[r, mapsto]
    & {a_1}  \ar[r, mapsto]
    & {a_2}  \ar[r, mapsto]
    & {a_3}  \ar[r, mapsto]
    & {\cdots}  \ar[r, mapsto]
    & {a_i} 
    \\
	& {b_0}  \ar[r, mapsto]
    & {b_1}  \ar[r, mapsto]
    & {b_2}  \ar[ur, mapsto]\\
	{G:} 
    & {c_0} \ar[r, mapsto]
    & {c_1} \ar[r, mapsto] 
    & {c_2} \ar[r, mapsto] 
    & {c_3} \ar[r, mapsto] 
    & {\cdots} \ar[r, mapsto] 
    & {c_i} 
    \\
	& {d_0}  \ar[ur, mapsto]
    && {d_2} \ar[ur, mapsto]
\end{tikzcd}\]
The distance between the points in their respective sets are:
\begin{equation*}
    d_0^F(a_0,b_0) =  3 \qquad d_1^F(a_1,b_1) = 2 \qquad  d_2^F(a_2,b_2) = 1 \qquad d_0^G(c_0,d_0) =  1 \qquad d_2^G(c_2,d_2) = 1  
\end{equation*}

Consider the $\T_1$-assignment $\phi$  and $\psi$ given by:
\begin{align*}
    \phi_i(a_i) = c_{i+1}  \qquad  \phi_0(b_0) = c_1 \qquad  \phi_1(b_1) = d_2, \qquad \phi_2(b_2) = c_3\\
    \psi_i(c_i) = a_{i+1}  \qquad  \psi_0(d_0) = b_1  \qquad    \psi_2(d_2) = a_3
\end{align*}
Then, the non zero loss terms we have are 
$$
\Lpl^{0,1}(\phi) = \Ltu^{0}(\phi,\psi) = d_2^G(c_2,d_2) = 1  
\quad \text{and} \quad  
\Ltd^{0}(\phi,\psi) =  \Lpr^{0,1}(\phi) = d_2^F(a_2,b_2) =  1. 
$$
So $L(\phi,\psi) = 1$.
\end{example}

\subsection{Bounding the interleaving distance}
\label{ssec:BoundingInterleaving}

In this subsection, we show how a $\T_\e$-assignment with loss $\delta $ induces a $\T_\delta\T_\e$-interleaving thereby bounding the interleaving distance by $\delta + \e$. 

\begin{definition}
Let $\cS \in \Flows_{\T}$  and take $(\phi,\psi)$ a $\cS$-assignment. 
Given a translation $\cR \in \Flows_\T$, the $\cR \cS$-assignment $(\phi',\psi')$  given by 
$\phi' = G[\cR] \circ \phi$
and 
$\psi' = F[\cR] \circ \psi$
is called  the \emph{$\cR$-translation of $(\phi,\psi)$}. 
\end{definition}

We next analyze how the $\cR$-translation transforms the loss of the assignment $(\phi,\psi)$.

\begin{theorem}
\label{th_L_functor}
    Fix $p,q$ in $P$ and $\cR,\cS \in \Flows_{\T}$. 
    Let  $(\phi,\psi)$ be a $\cS$-assignment and $(\phi',\psi')$ the $\cR$-translation of $(\phi,\psi)$. 
    Then, the following holds.
    \begin{enumerate}[nosep]
        \item $\Lpl^{p,q}(\phi') \leq  \Lpl^{p,q}(\phi)$
        \item $\Ltd^p(\phi',\psi') \leq  \max \{\Ltd^p(\phi,\psi) , \Lpr^{\cS p , \cR \cS p }(\psi) \}$
        \item $\Lpr^{p,q}(\psi') \leq  \Lpr^{p,q}(\psi) $
        \item $\Ltu^p(\phi',\psi') \leq  \max \{\Ltu^p(\phi,\psi) , \Lpl^{\cS p , \cR \cS p }(\phi) \}$.
    \end{enumerate}
\end{theorem}
    
\begin{proof}
Note that (1) and (3); and (2) and (4) are symmetric. 
So we will only prove statements (1) and (2). 

We start with statement (1). 
Fix $a \in F(p)$. 
Consider the elements 
$b = \phi_q \circ F[p \leq q](a)$,  
$b'= G[\cS p \leq \cS q] \circ \phi_p(a')$, 
$c = \phi'_q \circ F[p \leq q] (a)$ and 
$c' =  G[\cR\cS p \leq \cR\cS q] \circ \phi'_p(a')$. 
Such elements can be seen in the following non-commutative diagrams:

\begin{minipage}{0.45\textwidth}
    \[\begin{tikzcd}[cramped]
        {F(p)} && {F(q)} \\
        & {G(\cS p)} && {G(\cS q)} \\
        & {G(\cR\cS p)} && {G(\cR\cS q)}
        \arrow["{{F[\hspace{.5ex}\leq \hspace{.5ex}]}}", from=1-1, to=1-3]
        \arrow["{{\phi_p}}"', from=1-1, to=2-2]
        \arrow["{\phi'_p}"{description}, curve={height=12pt}, from=1-1, to=3-2]
        \arrow["{{\phi_q}}", from=1-3, to=2-4]
        \arrow["{\phi'_q}"{description}, curve={height=12pt}, from=1-3, to=3-4]
        \arrow["{{G[\; \leq \;]}}"'{pos=0.4}, from=2-2, to=2-4]
        \arrow["{{G[\cR]}}", from=2-2, to=3-2]
        \arrow["{{G[\cR]}}", from=2-4, to=3-4]
        \arrow["{{G[\; \leq \;]}}"', from=3-2, to=3-4]
    \end{tikzcd}\]
\end{minipage}
  \renewcommand{\arraystretch}{1}
\begin{minipage}{0.45\textwidth}
    \[\begin{tikzcd}[cramped,row sep = 10pt]
        a && \bullet \\
        & \bullet && \begin{array}{c} b \\ b' \end{array}  \\
        & \bullet && \begin{array}{c} c \\ c' \end{array}
        \arrow[ maps to, from=1-1, to=1-3]
        \arrow[ maps to, from=1-1, to=2-2]
        \arrow[curve={height=12pt}, maps to, from=1-1, to=3-2]
        \arrow[ maps to, from=1-3, to=2-4]
        \arrow[curve={height=12pt}, pos=0.6,maps to, from=1-3, to=3-4]
        \arrow[shift right=2, maps to, from=2-2, to=2-4]
        \arrow[maps to, from=2-2, to=3-2, dashed]
        \arrow[maps to, from=2-4, to=3-4, dashed]
        \arrow[shift right=2, maps to, from=3-2, to=3-4]
    \end{tikzcd}\]
\end{minipage}

Since by definition $\phi' = G[\cR] \circ \phi$, 
the left and right triangles commute so  
$c = G[\cR]\left(b\right)$ 
and 
$c' = G[\cR] \left(b' \right)$. 
By Lemma $\ref{lemma_preimage}$ we have that 
    \begin{equation*}
        d_{\cR\cS q}^G(c,c') \leq d_{\cS q}^G(b,b').
    \end{equation*}
Therefore, by definition of the parallelogram loss  and the properties of supremum we have 
$    \Lpl^{p,q}(\phi') \leq  \Lpl^{p,q}(\phi)$, 
finishing statement (1). 

For statement (2), let $\cR \cS = \cS' $ and consider the following non necessarily commutative diagram.

\begin{minipage}{0.6\textwidth}
    \[
    \adjustbox{scale=.8,center}{%
    \begin{tikzcd}[cramped]
        {F(p)} && {F(\cS\cS p)} & {F(\cS' \cS p)} & {F(\cS' \cS' p)} \\
        & {G(\cS p)} && {F(\cS \cS' p)} \\
        && {G(\cS' p)}
        \arrow["{F[\; \leq \; ]}", from=1-1, to=1-3]
        \arrow["{{{\phi_p}}}", from=1-1, to=2-2]
        \arrow["{{\phi'_p}}"{description}, curve={height=30pt}, from=1-1, to=3-3]
        \arrow["{F[\cR ]}", from=1-3, to=1-4]
        \arrow["{F[\cR ]}", from=1-4, to=1-5]
        \arrow["{{||}}"{description}, draw=none, from=1-4, to=2-4]
        \arrow["{{{\psi_{\cS p}}}}"', from=2-2, to=1-3]
        \arrow["{{G[\cR ]}}", from=2-2, to=3-3]
        \arrow["{F[\cR ]}"'{pos=0.3}, from=2-4, to=1-5]
        \arrow["{{{\psi'_{\cS' p}}}}"{description}, curve={height=30pt}, from=3-3, to=1-5]
        \arrow["{{{\psi_{\cS' p}}}}"', from=3-3, to=2-4]
    \end{tikzcd}
    }
    \]
\end{minipage}
\hfill
\begin{minipage}{0.35\textwidth}
    \[
    \adjustbox{scale=.8,center}{
    \begin{tikzcd}[cramped,row sep = 10pt,column sep = 16pt]
        a && \begin{array}{c} \begin{array}{c} b \\ b' \end{array} \end{array} & \begin{array}{c} \begin{array}{l} \\c \\ c' \\ c'' \end{array} \end{array} & \begin{array}{c} \begin{array}{c} \\ e \\ e' \\ e'' \end{array} \end{array} \\
        & \bullet \\
        && \bullet
        \arrow[shift left=2, maps to, from=1-1, to=1-3]
        \arrow[maps to, from=1-1, to=2-2]
        \arrow[curve={height=30pt}, maps to, from=1-1, to=3-3]
        \arrow[shift right=2, maps to, from=1-3, to=1-4]
        \arrow[shift left=3, maps to, from=1-3, to=1-4]
        \arrow[shift right=2, maps to, from=1-4, to=1-5]
        \arrow[shift right=7, maps to, from=1-4, to=1-5]
        \arrow[shift left=3, from=1-4, to=1-5]
        \arrow[maps to, from=2-2, to=1-3]
        \arrow[maps to, from=2-2, to=3-3]
        \arrow[shift right=2, maps to, from=3-3, to=1-4]
        \arrow[curve={height=30pt}, maps to, from=3-3, to=1-5]
    \end{tikzcd}
    }\]
\end{minipage}

For $a \in F(p)$ set 
$e = F[p \leq \cS \cS' p](a)$ and 
$e''= \psi'_{\cS' p} \circ \phi'_p (a)$. 
Since by Thm.~\ref{th_ultrametric} the merging distance is an ultrametric, we have that 
$$
d_{\cS \cS' p}^F(e,e'') \leq \max\left\{ d_{\cS' \cS' p}^F(e,e') , d_{\cS' \cS' p}^F(e',e'')\right\}.
$$

By Lemma \ref{lemma_preimage},  $ d_{\cS' \cS' p}^F(e,e') \leq d_{\cS' \cS p}^F(c,c') \leq d_{\cS \cS p}^F(b,b')$. 
On the other hand, by the same Lemma,  $ d_{\cS' \cS' p}^F(e',e'') \leq  d_{\cS \cS' p}^F(c',c'')$. 
Hence
$$
d_{\cS \cS' p}^F(e,e'') \leq \max\left\{ d_{\cS \cS p}^F(b,b') , d_{\cS \cS' p}^F(c',c'')\right\}.
$$
By definition of the loss of the parallelogram and triangle diagrams and the definition of supremum, we conclude that 
$$ 
\Ltd^p(\phi',\psi') \leq  \max \{\Ltd^p(\phi,\psi) , \Lpr^{\cS p , \cR \cS p }(\psi) \}
$$
finishing statement (2) and the proof.
\end{proof}

Intuitively, Theorem \ref{th_L_functor} says that the loss is non-increasing when translating by $\cR$. Further, we can reach a zero-loss for a special $\T_\delta$- translation.

\begin{lemma}
\label{lemma_de-interleaving}
     Given a  a $\T_\e$-assignment $(\phi,\psi)$, choose a $\delta$ such that 
     $$
     \delta > \max\{\Lpl^{p,q}(\phi), \Lpr^{p,q}(\psi),  \Ltd^p(\phi,\psi) , \Lpr^{\T_\e p , \T_\delta \T_\e p }(\psi), \Ltu^p(\phi,\psi) , \Lpl^{\T_\e p , \T_\delta \T_\e p }(\phi)  \}.
     $$
     Define $(\Phi,\Psi)$ to be the $\T_\delta$-translation of $(\phi,\psi)$. 
     Then 
     $$
     \Lpl^{p,q}(\Phi) = 0, \qquad 
     \Lpr^{p,q}(\Psi) = 0, \qquad 
     \Ltd^p(\Phi,\Psi) =0, \qquad \text{and} 
     \qquad \Ltu^p(\Phi,\Psi) =0.
     $$
\end{lemma}
 \begin{proof}
 For the parallelogram loss consider the proof of  Statement (1) in Theorem \ref{th_L_functor}. 
 Since 
 $$
 \delta > \Lpl^{p,q}(\phi) \geq d_{\T_\e q}^G(b,b')
 $$ 
 then by Cor.~\ref{coro_slide}, 
 $d_{\T_\delta\T_\e q}^G(c,c') = 0$. 
 Hence $\Lpl^{p,q}(\Phi) = 0$. 
 The symmetric argument gives that $\Lpl^{p,q}(\Psi) = 0$. 

 For the triangle loss consider the proof of Statement (2) in Thm.~\ref{th_L_functor}. 
 Using Cor.~\ref{coro_slide}, we have that for any $a \in F(p)$, 
 $d_{\T_\delta \T_\e \T_\delta \T_\e p}^F(e,e') = 0$ given that  
 $$
 \delta > \Ltd^p(\phi,\psi) \geq d_{\T_\e \T_\e p}^F(b,b'). 
 $$ 
 Similarly $d_{\T_\delta \T_\e \T_\delta \T_\e p}^F(e',e'')=0$ provided that  
 $$
 \delta > \Lpr^{ \T_\e p , \T_\delta \T_\e p }(\psi) \geq d_{T_\delta \T_\e \T_\e p}^F(c',c'').
 $$
 Therefore, $d_{\T_\delta \T_\e \T_\delta \T_\e p}^F(e,e'') = 0$, so, $\Ltd^p(\Phi,\Psi) = 0$. 
 Again symmetrically, $\Ltd^p(\Phi,\Psi) = 0$.
 \end{proof}

\begin{theorem}
\label{th_bounding}
   For any  $\T_\e$-assignment given by $\phi: F \rightdsar G\T_\e$ and $\psi: G \rightdsar F\T_\e$, 
   \begin{equation*}
       d_I(F,G) \leq \e + L(\phi, \psi).
   \end{equation*}
\end{theorem}
 \begin{proof}
    Set $\delta > L(\psi, \phi)$. By Lemma \ref{lemma_de-interleaving} the $\T_\delta \T_\e$-assignment $(\Phi, \Psi)$ given by  $\Phi = G[\T_\delta ] \circ \phi$  and $ \Psi = F[\T_\delta] \circ \psi $ is a $\T_\delta \T_\e$-interleaving of $F$ and $G$. By monotonicity of the interleaving we have that $F$ and $G$ are $\T_{\delta +\e}$-interleaved. Therefore, as $\delta$ is arbitrary, we have that  $d_I(F,G) \leq \e + L(\psi, \phi)$ as we wanted.
\end{proof}

\begin{example}
\label{example:InterleavingBound}
    Take $F,G$  and $\T$ from Example and \ref{example_F_G}. 
    For $\e = 1$, we calculated $L(\phi,\psi) = 1$, so by Theorem \ref{th_bounding} we obtain that  $d_I(F,G) \leq 2$.
\end{example}

We note that this theorem makes no promise of tightness, and indeed we would not expect that. 
Instead, it can be viewed as providing a bound for any input assignment, and then the goal would be to find an assignment providing the best possible bound under some circumstances. 
However, we leave that particular direction to future work and instead focus on determining what is necessary to compute $L(\phi,\psi)$ for a given input assignment.

\section{Reducing calculations}

\label{sec_improve}

This section is devoted to creating algorithms and strategies to calculate  $L(\phi,\psi)$ using different properties of the distance function and by narrowing $P$ and $\C$ to particular cases largely of interest to the TDA community. 
Throughout the section, we assume that $F,G \in \C^P$, $\T$ is a commutative flow and and $(\phi,\psi)$ is a $\mathcal{S}$-assignment between $F$ and $G$, unless said explicitly otherwise. 

\subsection{Improving the calculation of the merging distance function}

\begin{definition}
     For a fixed $q \in P$, we say that $\delta \in [0,\infty)$ is a \emph{reducing constant} of $q$ on $F$ if for all 
     $0 \leq \gamma_1 < \delta \leq \gamma_2$, with $\gamma_1 \neq \gamma_2$, 
     $F[\T_{\gamma_1} q \leq \T_{\gamma_2} q]$ is non-monomorphism. 
     Otherwise, we say that $\delta$ is a \emph{non-reducing constant} of $q$ on $F$.  
\end{definition}

For $q \in P$, we write  $D^F_q\subset \R_{\geq 0}$ for the set of reducing constants of $q$ on $F$. 

\begin{example}
\label{example:reducing_Constants}
Continuing with  the flow $\T$ and the module $F$ of Example \ref{example_F_G},  consider the point $q=1\in \N$. 
We will show that the set of reducing constants of $1$ on $F$ is $D_1^F = \{2\}$.
Note that for $0<\gamma_1 < 2 $ we have $\T_{\gamma_1}(1) < 2$. 
Meanwhile, for $\gamma_2 \geq 2$ we obtain $\T_{\gamma_2}(1) \geq 3$. 
Then for such values, 
$$
F[\T_{\gamma_1}(1) \leq \T_{\gamma_2} (1)]
= F[1 + \floor{\gamma_1} \leq 2 ] \circ F[2 \leq 3] \circ F[3 \leq 1 + \floor{\gamma_2}] 
$$  
Hence, this morphism is not a monomorphism given that  $F[2 \leq 3]$ is not injective. Moreover, note that this in the only non injective morphism of the form $F[i \leq i+1]$, and so there are no other elements in $D_1^F$ for this example. 
\end{example}

Our next lemma shows that the the distance can only take values in this set of reducing constants.

\begin{lemma}
    \label{lemma_reducing_constants}
    Fix $q \in P$. 
    If  $D^F_q$ is finite,  then for any  $a \neq b$ in $F(q)$, $d_q^F(a,b) \in D^F_q \cup \{\infty \}$.
\end{lemma}
\begin{proof}
Let $D :=D_q^F = \{ \tuple{\delta}{m}\}$ be sorted so that $\delta_i < \delta_{i+1}$. 
Seeking a contradiction, say there is a pair $a\neq b$ in $F(q)$ and an  $i$ with  
$\delta_{i-1} < d_q^F(a,b) < \delta_i$ (setting $\delta_0 = 0$ if need be). 
Then, for  any $\delta$ with $d_q^F(a,b) \leq \delta < \delta_{i}$, we have that 
$ F[q \leq \T_{\delta} q](a) =  F[q \leq \T_{\delta} q](b)$. 
Since $\delta \notin D$, choose $0 < \gamma < \delta$ such that $F[\T_{\gamma} q \leq \T_{\delta} q]$ is a monomorphism. 
Recall that $F[q \leq \T_{\delta} q] = F[\T_{\gamma} q \leq \T_{\delta} q] \circ F[q \leq \T_{\gamma} q] $. 
This implies that $F[q \leq \T_{\gamma} q](a) =  F[q \leq \T_{\gamma} q](b)$, then $d_q^F(a,b) \leq \gamma $. Since $\delta$ was arbitrary we obtain that   
$$
F[q \leq \T_{d_q^F(a,b)} q](a) =  F[q \leq \T_{d_q^F(a,b)} q](b).
$$ 
Then, using a similar analysis for $\delta = d_q^F(a,b) < \delta_i$ we can find $\gamma < d_q^F(a,b)$ with $F[q \leq \T_{\gamma} q](a) =  F[q \leq \T_{\gamma} q](b)$, which is a contradiction. Hence, having a $\delta$ between $d_q^F(a,b)$ and $\delta_i$ yields a contradiction. Therefore,  $ \delta_i = d_q^F(a,b)$.

Now, consider the case where $ \delta_n \leq  d_q^F(a,b)$. 
If $d_q^F(a,b) = \infty$ we are done. 
If not, note that for 
$ \delta_n < \delta < d_q^F(a,b)$ 
we can find $\delta < \gamma \leq d_q^F(a,b)$ where $F[\T_{\delta} q \leq \T_{\gamma} q]$ is a monomorphism. 
By a similar process, we can obtain a $\gamma > d_q^F(a,b)$ with 
$F[q \leq \T_{\gamma} q](a) \neq  F[q \leq \T_{\gamma} q](b)$, which is a contradiction. 
Hence $\delta_n = d_q^F(a,b)$, finishing the proof.
\end{proof}

With this lemma, we have the following immediate theorem, which can limit the possible values of the loss function. 

\begin{theorem}
    \label{thm_finite_reducing}
    The following hold:
    \begin{enumerate}[1),nosep]
        \item  If $q \in P$ is such that  $D^G_{ \cS q}$ is finite, then for any $p \leq q$, $\Lpl^{p,q}(\phi) \in D^G_{ \cS q} \cup \{\infty\}$. 
        
        \item If $q \in P$ is such that  $D^F_{ \cS \cS  q}$ is finite, then $\Ltd^{q}(\phi,\psi) \in D^F_{ \cS \cS   q} \cup \{\infty\} $.
        
        \item If $q \in P$ is such that  $D^F_{ \cS q}$ is finite, then for any $p \leq q$, $\Lpr^{p,q}(\phi) \in D^F_{ \cS q} \cup \{\infty\}$. 
        
        \item If $q \in P$ is such that  $D^G_{ \cS \cS  q}$ is finite, then $\Ltu^{q}(\phi,\psi) \in D^G_{ \cS \cS  q} \cup \{\infty\} $.
    \end{enumerate}
\end{theorem}
Note that in many cases we are interested in (such as $q$-tame persistence modules in $\Vec^P$; or when $P$ is a finite set; or when the underlying sets of objects of $\C$ are finite) the set of reducing constants will always be finite. 

\subsubsection{Calculating loss in FC categories}

We will use Thm.~\ref{thm_finite_reducing} to describe different algorithms to calculate $\Lpl^{p,q}(\phi)$ where $p$ and $q$ are fixed. 
We assume throughout this section that $C$ is a concrete category where the underlying sets of objects are finite; that is, an FC category. 
With these assumptions, we can present an algorithm to compute $\Lpl^{p,q}(\phi)$, and note that the algorithm to compute $\Lpr^{p,q}(\psi)$, $\Ltd^p(\phi,\psi )$ and $\Ltu^p(\phi,\psi )$ are similar.

In what follows, we fix  $p < q \in P$ and write $D^G_{ \cS q} = \{ \tuple{\delta}{m}\}$. 
We fix a set of interleaved values $\{\gamma_i\}_{i=0}^m\subseteq \R_{\geq 0}$ with 
$0 = \gamma_0 <\delta_0 <\gamma_1 < \cdots <\delta_i < \gamma_i < \delta_{i+1} < \cdots \delta_m < \gamma_m$. 
To simplify notation in the algorithm, we write the composition morphisms  on the two sides of the parallelogram by 
$f = \phi_q \circ  F[p \leq q]$ 
and
$g = G[ \cS p \leq \cS q] \circ \phi_p$, so we have 
$\Lpl^{p,q}(\phi) = d(f,g)$.
The idea is based on a binary search (See e.g.~\cite{Erickson2019}), where at each step checking $\gamma_i$, we only keep the elements of $F(p)$ which do not map to the same element at the end of diagram $\Parallelograml_{p,q}(\phi)$ in $G(\T_{\delta_i} \cS q) $.

\begin{algorithm}[ht] %
\caption[Algorithm to compute Lpq]{Algorithm to compute $\protect\Lpl^{p,q}(\phi)$ in finite categories}
\label{alg:computation}

\begin{algorithmic}[1]
    \State \textbf{Given:} $D_{\cS q}^G = \{\delta_1 < \cdots < \delta_m \}$ with interleaved $\{\gamma_i\}_{i=0}^m$
    \State Set $L \gets 0$, $H \gets m+1$, and $A \gets F(p)$.
    \While{$L \neq H$}
        \State $i \gets \lfloor\frac{H+L}{2}\rfloor$
        \State $A' \gets \left\{a \in A \mid G[ \cS q < \T_{\gamma_i} \cS q ](f(a)) \neq G[ \cS q < \T_{\gamma_i}  \cS q ](g(a))\right\}$ 
        \If{$A' = \emptyset$}
            \State $H \gets i$
        \Else
            \State $L \gets i+1$
            \State $A \gets A'$
        \EndIf
    \EndWhile
    \If{$0 < L \leq m$}
        \State \Return $\delta_L$ \Comment{$\Lpl^{p,q}(\phi) = \delta_L$}
    \ElsIf{$L = 0$}
        \State \Return $0$ \Comment{$\Lpl^{p,q}(\phi) = 0$}
    \Else
        \State \Return $\infty$ \Comment{$\Lpl^{p,q}(\phi) = \infty$, as $L=m+1$}
    \EndIf
\end{algorithmic}
\end{algorithm}

The basic idea is to do a binary search on the set $D_{\cS q}^G$ to find the lowest $\delta_i$ for which all $ a \in F(p)$ have the same images $f(a)$ and $g(a)$ in $G(\T_{\delta_i} \cS q)$.
This means that the time complexity of calculating $\Lpl^{p,q}(\phi)$ is $O(n \log(m))$ where $n = |F(p)|$ and $m = |D_{\cS q}^G|$.

Similar computations show that the complexity of determining $\Ltd^p(\phi,\psi )$ is 
$O( |F(p)| \log |D_{\cS \cS p}^F|)$, for $\Lpr^{p,q}(\psi )$ is 
$O( |F(p)| \log |D_{\cS q}^F|)$ and for $\Ltu^p(\phi,\psi )$ is 
$O( |F(p)| \log |D_{\cS \cS p}^G|)$.

\subsubsection{Calculating the loss for vector fields}
We now turn to the case where there is additional structure on $\C$ to be used, in particular, the case where $\C = \Vec_\F$ for some field $\F$. 
For this portion, consider 
$F,G \in \Vec_\F^P$ to be $q$-tame modules.
As before, we we fix  $p < q \in P$ and write $D^G_{ \cS q} = \{ \tuple{\delta}{m}\}$ with interleaved values $\{\gamma_i\}_{i=0}^m\subseteq \R_{\geq 0}$. 

Recall that for any linear transformations, $h,h': V \to W$, we have that  $h(v) = h'(v)$ if and only if $v \in \ker(h -h')$.  Further, $\ker(h - h') = V$ if and only if $\im(h - h') = \{0\}$. 

The linear maps $f$ and $g$ are identical to those from the previous section. The following algorithm find the loss of the diagram based again on binary search. This time, it checks  $\gamma_i$ to find the lowest $\delta_i$ such that \mbox{$G[\cS q \leq  \T_{\delta_i} \cS q](Im(f-g)) = \{0\}$}.

\begin{algorithm}[ht]
\caption[Algorithm to compute Lpq in VecF]{Algorithm to compute $\protect\Lpl^{p,q}(\phi)$ in $\Vec_{\F}$}
\label{alg:lpl_calculation}
\begin{algorithmic}[1]
    \State Set $V \gets \im ( f - g ) \subseteq G( \cS q)$.
    \If{$V = \{0\}$}
        \State \Return $0$ \Comment{$\Lpl^{p,q}(\phi) = 0$}
    \Else
        \State Set $L \gets 1$, $H \gets m+1$
    \EndIf
    
    \While{$L \neq H$}
        \State $i \gets \lfloor\frac{H+L}{2}\rfloor$
        \State $V' \gets \im (M |_V)$ for linear map $M = G[\T_{\gamma_{L-1}} \cS q < \T_{\gamma_i} \cS q ]$
        \If{$V' = \{0\}$}
            \State $H \gets i$
        \Else
            \State $L \gets i+1$
            \State $V \gets V'$
        \EndIf
    \EndWhile
    
    \If{$L \leq m$}
        \State \Return $\delta_L$ \Comment{$\Lpl^{p,q}(\phi) = \delta_L$}
    \Else
        \State \Return $\infty$ \Comment{$\Lpl^{p,q}(\phi) = \infty$, as $L=m+1$}
    \EndIf
\end{algorithmic}
\end{algorithm}

For this case, the  time complexity of calculating $\Lpl^{p,q}(\phi)$ is $O(n^3\log(m))$ where $n = \rk(f - g)$ as we are searching the lowest $\delta_i$ such that $Im(f-g)$ is map to 0 in $G(\T_{\delta_i} \cS q)$.
  
As before, we can use a similar methodology  to calculate $\Lpr^{p,q}(\psi)$, $\Ltd^p(\phi,\psi )$ and $\Ltu^p(\phi,\psi )$ using the corresponding restrictions of the reducing constants of $ \cS q$ on $F$ and $ \cS \cS p$ on $F$ and $G$ respectively.

\subsection{Reducing the number of diagrams}

In the last section we optimized the calculation of $\Lpl^{p,q}(\phi)$, $\Lpr^{p,q}(\psi)$, $\Ltd^p(\phi,\psi )$ and $\Ltu^p(\phi,\psi )$ for fixed $p$ and $q$ under certain conditions. 
Next, we will reduce the number of diagrams we need to check to calculate $L(\phi,\psi)$.

Given $p < r < q$ in $P$, we begin by establishing bounds on the parallelogram loss corresponding to different pairs of elements.
\begin{lemma}
    \label{lemma_lll}
    Let $p < r < q$ in $P$. Then the following holds:
    \begin{enumerate}[1)]
        \item \begin{equation*}
        \Lpl^{p,q}(\phi) \leq \max \left\{ \Lpl^{p,r}(\phi) , \Lpl^{r,q}(\phi) \right\}.
    \end{equation*}
    \item If $F[p \leq r]$ is an isomorphism, then 
    \begin{equation*}
        \Lpl^{r,q}(\phi) \leq \max \left\{ \Lpl^{p,r}(\phi) , \Lpl^{p,q}(\phi) \right\}.
    \end{equation*}
    \item If $\T$ is an Archimedean flow (Defn.~\ref{def:Archimedean}), there is $\delta > 0$ such that  
    \begin{equation*}
        \Lpl^{p,r}(\phi) \leq \max \left\{ \Lpl^{p,q}(\phi) , \Lpl^{r,q}(\phi) \right\} + \delta
    \end{equation*}
    \end{enumerate}
 \vspace{-1em}
 \end{lemma}
 
 \begin{proof}
    For $a \in F(p)$ consider the following elements:
\[
\adjustbox{scale=.9,center}{
\begin{tikzcd}[cramped]
    F(p) & F(r) & F(q) \\
    & G(\cS p) & G(\cS r)  & G(\cS q) 
	\arrow["{F[p \leq r]}",  from=1-1, to=1-2]
	\arrow["{\phi_p}"',  from=1-1, to=2-2]
	\arrow["{F[r \leq q]}", from=1-2, to=1-3]
	\arrow["{\phi_r}",  from=1-2, to=2-3]
	\arrow["{\phi_q}"',  from=1-3, to=2-4]
	\arrow["{G[ \; \leq \; ]}"', shift right=2,  from=2-2, to=2-3]
    \arrow["{G[ \; \leq \; ]}"', shift right=2,  from=2-3, to=2-4]
\end{tikzcd}
\qquad
\begin{tikzcd}[cramped]
    a && {a'} && \bullet \\
    & \bullet & {} & \begin{array}{c} b' \\ b \end{array} && \begin{array}{c} c'' \\ c' \\ c \end{array}
	\arrow["{F[p \leq r]}", maps to, from=1-1, to=1-3]
	\arrow["{\phi_p}"', maps to, from=1-1, to=2-2]
	\arrow["{F[r \leq q]}", from=1-3, to=1-5]
	\arrow["{\phi_r}", maps to, from=1-3, to=2-4]
	\arrow["{\phi_q}"', maps to, from=1-5, to=2-6]
	\arrow["{G[ \cS p \leq  \cS r]}"', shift right=2, maps to, from=2-2, to=2-4]
	\arrow[ shift left=2, bend right=22,looseness=0.8, maps to, from=2-4, to=2-6]
    \arrow["{G[\cS  r \leq \cS q]}"', shift right=4, curve={height=6pt}, maps to, from=2-4, to=2-6]
\end{tikzcd}
}
\]
By Thm.~\ref{th_ultrametric} we know that $d^G_{\cS q}$, then,
$
    d_{ \cS q}^G (c,c'') \leq \max \left\{  d_{ \cS q}^G (c,c') ,  d_{ \cS q}^G (c',c'')  \right\}.
$
One one hand the rightmost square is $\Parallelograml_\phi(r,q)$, so the we have that  $d_{ \cS q}^G (c',c'') \leq \Lpl^{r,q}(\phi)$. 
On the other hand, by Lemma \ref{lemma_preimage}, 
$d_{ \cS q}^G (c,c') \leq  d_{\cS r}^G (b,b')$, and because the left square is $\Parallelograml_\phi(p,r)$ implying 
$d_{\cS r}^G (b,b') \leq  \Lpl^{p,r}(\phi)$, we have
$d_{ \cS q}^G (c,c') \leq \Lpl^{p,r}(\phi) $.
Putting these together, we have 
\begin{equation*}
    d_{\cS q}^G (c,c'') \leq \max \left\{  \Lpl^{r,q}(\phi)  ,  \Lpl^{p,r}(\phi) \right\}
\end{equation*}
and by definition of supremum, statement 1) follows. 

For statement 2), assume $F[p \leq r]$ is an isomorphism and consider $a' \in F(r)$. 
Choose $a = F[p \leq r]^{-1}(a')$. 
By the same reasoning as above the conclusion follows because we have that
$
    d_{\cS q}^G (c',c'') \leq \max \left\{  d_{\cS q}^G (c',c) ,  d_{\cS q}^G (c,c'')  \right\} 
$

Finally for statement 3), assume that $\T$ is an Archimedean flow. Since $c = G[\cS r \leq \cS q](b)$ and $c' = G[\cS r \leq \cS q](b')$ by Lemma \ref{lemma_preimage_Archimedean} there is $\delta$ such that $d_{\cS r}^G (b,b') < d_{ \cS q}^G (c,c') + \delta$.  Hence, because $d^G_{\cS q}$ is an ultrametric,
\vspace{-1pt}
\[d_{\cS r}^G (b,b') \leq \max \left\{  d_{\cS q}^G (c,c'') ,  d_{\cS q}^G (c'',c')  \right\} + \delta \] 

It implies that  
$
    d_{\cS r}^G (b,b') \leq \max \left\{ \Lpl^{p,q}(\phi)  ,  \Lpl^{r,q}(\phi) \right\}  + \delta 
$. Hence the statement 3) follows.
\end{proof}

To obtain inequalities that involve the triangle diagram we need a specific properties of the flow

\begin{lemma}
    \label{lemma_ll_triang}
    Let $p < q$ in $P$ such that $\Lpl^{p,q}(\phi) = 0$ and  $\Lpr^{ p, q}(\psi) = 0$. 
    \begin{enumerate}[(1)]
        \item 
    If $F[p \leq q]$ is isomorphism, we have that
    \begin{equation*}
        \Ltd^q(\phi,\psi) \leq \Ltd^p(\phi,\psi).
    \end{equation*}
    \item 
    If $\T$ is an Archimedean flow (Defn.~\ref{def:Archimedean}), there is $\delta > 0$ such that  
    \begin{equation*}
        \Ltd^p(\phi,\psi) <\Ltd^q(\phi,\psi) + \delta.
    \end{equation*}
    \end{enumerate}
    \end{lemma}
 \begin{proof}
 Without loss of generality, assume that $ q \leq   \cS \cS p$.
 For the sake of brevity let $p'=  \cS \cS p$ and $q' =  \cS \cS  q$ and consider the following diagram:
    \[
    \adjustbox{scale=.8,center}{
    \begin{tikzcd}[cramped]
	F(p) && F(q) && F(p') && F(q') \\
	&& G( \cS p) && G( \cS q)
	\arrow["{F[p\leq q]}",  from=1-1, to=1-3]
	\arrow["{\phi_p}"'{pos=0.3},  from=1-1, to=2-3]
	\arrow["{F[ q \leq p' ]}",   from=1-3, to=1-5]
	\arrow["{\phi_q}"'{pos=0.3},   from=1-3, to=2-5]
	\arrow["{F[ p' \leq q' ]}",   from=1-5, to=1-7]
	\arrow["{\psi_{ p}}"'{pos=0.8},  from=2-3, to=1-5]
	\arrow[shorten <=5pt, from=2-3, to=2-5]
	\arrow["{\psi_{ q}}"'{pos=0.7},   from=2-5, to=1-7]
    \end{tikzcd}
    \qquad
    \begin{tikzcd}[cramped]
	a && a' && \begin{array}{c} \phantom{b} \\b \\ b' \end{array} && \begin{array}{c} \phantom{c} \\c \\ c' \end{array} \\
	&& \bullet && \bullet
	\arrow["{F[p\leq q]}", mapsto, from=1-1, to=1-3]
	\arrow["{\phi_p}"'{pos=0.3}, mapsto, from=1-1, to=2-3]
	\arrow["{F[ q \leq p' ]}", mapsto,  from=1-3, to=1-5]
	\arrow["{\phi_q}"'{pos=0.3}, mapsto,  from=1-3, to=2-5]
	\arrow["{F[ p' \leq q' ]}", mapsto,  from=1-5, to=1-7]
	\arrow[shift right=5, from=1-5, to=1-7]
	\arrow["{\psi_{ p}}"'{pos=0.8}, mapsto, shift right=2, from=2-3, to=1-5]
	\arrow[shorten <=5pt, from=2-3, to=2-5]
	\arrow["{\psi_{ q}}"'{pos=0.7}, mapsto, shift right=2, shorten >=5pt, from=2-5, to=1-7]
\end{tikzcd}
}
\]
Given $a' \in F(q)$, and because $F[p \leq q]$ is an isomorphism, we can choose $a = F[p \leq q]^{-1}(a')$.
Because the left triangle is $\triangled_{\phi,\psi}(p)$, we have that $d_{p'}^F (b',b) \leq \Ltd^p(\phi,\psi)$. 
In addition, $d_{q'}^F (c,c') \leq  d_{p'}^F (b',b)$ by Lemma \ref{lemma_preimage}, so 
$d_{q'}^F (c,c') \leq \Ltd^p(\phi,\psi)$. 
By definition of supremum it follows that $\Ltd^q(\phi,\psi) \leq \Ltd^p(\phi,\psi)$ finishing statement (1).
    
For statement (2), assume that $\T$ is an Archimedean flow and choose any $a \in F(p)$. By Lemma ~\ref{lemma_preimage_Archimedean} given that $c = F[p' \leq q'](b)$ and $c' = F[p' \leq q'](b')$ there is $\delta$ such that  $d_{p'}^F(b,b') < d_{q'}^F(c,c') + \delta $, which then implies that $ \Ltd^p(\phi,\psi) < \Ltd^q(\phi,\psi) + \delta$. 
\end{proof}

\begin{corollary}
\label{cor:LossDiscrete}
    If $P$ is discrete,
    \begin{equation*}
        L(\phi,\psi) = \sup_{p \prec q} \left\{  \Lpl^{p,q}(\phi),  \Ltd^q(\phi,\psi ), \Lpr^{p,q}(\psi) , \Ltu^q(\phi,\psi )\right\}.
    \end{equation*}
\end{corollary}
 \begin{proof}
 Fix $p \prec p'$ and $q' \prec q$.
 By Lem.~\ref{lemma_lll},  we have
 $\Lpl^{p,q}(\phi) \leq \max \left\{ \Lpl^{p,p'}(\phi) , \Lpl^{q',q}(\phi) \right\}$.
The triangles are not affected since there is no pair involved in the diagram, so the corollary follows. 
\end{proof} 

Consider $P$ a finite poset. 
For a fixed $q \in P$, Cor.~\ref{cor:LossDiscrete} implies that the number of parallelogram diagrams needed to check  is $d(P) \cdot |P|$ where $d(P)$ is the maximum number of predecessors of any element in $P$. 
This implies that the running time for computing $L(\phi,\psi)$ in the case of a finite poset $P$ and $\C$ a $FC$ category is  
\[O\left(|P| \; d(P) \; N \; log(|P|) \right)\] where $N  = \max_{p \in P} \{|F(P)|, |G(P)| \}$
and for $\C = \Vec_\F$ is
\[O\left(|P| \; d(P) \; N^3 \; log(|P|) \right)\] where $N^3  = \max_{p \prec q \in P} \{\rk F([p \prec q], \rk G([p \prec q], \rk \phi_p, \rk \psi_p \}$ recalling that ny linear transformations, $h,h': V \to W$, we have that $\rk (h -h') \leq \rk (v) + \rk(v')$.

Note that this number of diagrams to check cannot be further improved. 
In general, $p < q < r$ in $P$, lack of commutativity of assignments mean that $\Ltd^p(\phi,\psi )$ and $\Ltd^q(\phi,\psi )$ do not bound each other, unless we have some additional conditions like those in Lemma \ref{lemma_ll_triang}. 
Similarly $\Lpl^{p,r}(\phi )$, and $\Lpl^{r,q}(\phi )$ do not bound each other. 
This can be seen in the following example.

\begin{example}
    We continue with Example \ref{example_F_G} using the flow $\T$ given by $\T_\e = m + \floor{\e}$. 
    We take the same $F$ along with $\hat{G}$ (a  modification of $G$ from that example) in $\Set^{\N}$  given by
   \[\begin{tikzcd}[cramped, row sep = tiny]
	{F:} & {a_0} & {a_1} & {a_2} & {a_3} & {a_4} & {a_5} & \cdots & {a_i} \\
	& {b_0} & {b_1} & {b_2} \\
	&&&& {b'_3} & {b'_4} \\
	{\hat{G}:} & {c_0} & {c_1} & {c_2} & {c_3} & {c_4} & {c_5} & \cdots & {c_i} \\
	& && {c'_2} & {c'_3}
	\arrow[from=1-2, to=1-3]
	\arrow[from=1-3, to=1-4]
	\arrow[from=1-4, to=1-5]
	\arrow[from=1-5, to=1-6]
	\arrow[from=1-6, to=1-7]
	\arrow[from=1-7, to=1-8]
	\arrow[from=1-8, to=1-9]
	\arrow[from=2-2, to=2-3]
	\arrow[from=2-3, to=2-4]
	\arrow[from=2-4, to=1-5]
	\arrow[from=3-5, to=3-6]
	\arrow[from=3-6, to=4-7]
	\arrow[from=4-2, to=4-3]
	\arrow[from=4-3, to=4-4]
	\arrow[from=4-4, to=4-5]
	\arrow[from=4-5, to=4-6]
	\arrow[from=4-6, to=4-7]
        \arrow[from=4-7, to=4-8]
	\arrow[from=4-8, to=4-9]
	\arrow[from=5-5, to=4-6]
        \arrow[from=5-4, to=5-5]
\end{tikzcd}\]
The distance between the points in their respective sets are:
\begin{equation*}
\begin{matrix}
    d_0^F(a_0,b_0) =  3 & d_1^F(a_1,b_1) = 2 &  d_2^F(a_2,b_2) = 1 & d_3^{\hat{G}}(b'_3,c_3) = 2 \\
    && d_2^{\hat{G}}(c_2,c'_2) = 2  & d_3^{\hat{G}}(b'_3,c_3') = 2    & d_4^{\hat{G}}(b'_4,c_4) = 1\\
    &&&d_3^{\hat{G}}(c_3,c'_3) = 1 &
\end{matrix}
\end{equation*}
Let us consider the $\T_1$-assignment $\phi$  and $\psi$ between $F$ and $\hat{G}$ with $\phi$ given by
\begin{equation*}
\begin{array}{llll}
\multicolumn{3}{c}{\phi_i(a_i) = c_{i+1}, \; i  \in \{0,1,2\}} & \phi_3(a_3) = b'_4 \\
\phi_0(b_0) = c_1 &  \phi_1(b_1) = c'_2 & \phi_2(b_2) = b'_3 &  
\end{array}
\end{equation*}
and $\psi$ sends everything to $a$, i.e.~$\psi_i(\bullet_i) = a_{i+1}  $.
Then, some of the losses we have are
\begin{align*}
    &\Lpl^{0,1}(\phi) = \Ltu^{0}(\phi,\psi) = d_2^{\hat{G}}(c_2,c'_2) = 1  \qquad \Ltu^{1}(\phi,\psi) = 0  \qquad  \Lpl^{1,2}(\phi) =  d_2^{\hat{G}}(b'_3,c'_3) = 2 \\ 
    &\Lpl^{1,3}(\phi) = \Lpl^{2,3}(\phi) = \Ltu^{2}(\phi,\psi) = d_4^{\hat{G}}(c_4,b'_4) = 1 \qquad \Lpl^{0,2}(\phi) = d_2^{\hat{G}}(c_3,b'_3) = 2 
\end{align*} 
In this way we obtain that 
\begin{align*}
     1 = \Lpl^{0,1}(\phi) < \Lpl^{1,2}(\phi) = 2 \quad &\text{but} \quad 2 = \Lpl^{1,2}(\phi) > \Lpl^{2,3}(\phi) = 1, \qquad \text{and} \\
    1 = \Ltu^{0}(\phi,\psi) > \Ltu^{1}(\phi,\psi) = 0 \quad &\text{but} \quad 0 = \Ltu^{1}(\phi,\psi) < \Ltu^{2}(\phi,\psi) = 1 . 
\end{align*}
so we cannot somehow determine $\Lpl^{p, r}(\phi)$ from  $\Lpl^{r,q}(\phi)$ or vice versa; similarly for $\Ltu^{p}(\phi,\psi)$ and $\Ltu^{q}(\phi,\psi)$.

\end{example}

\subsection{Optimization over complete linear orders}
\label{ssec:Linear_orders_opt}

We assume in this section that $P$ is a complete linear order. 
Further, for simplicity we assume $(\phi,\psi)$ is a $\T_\e$-assignment instead of $\mathcal{S}$-assignment. 
While the results are still valid for any $\cS \in \Flows_{\T}$, we focus on $\T_\e $ as we are interested in cases inducing interleavings from $\T_\e$-assignments. 

Note that to obtain reduction formulas like in Corollary \ref{cor:LossDiscrete}, we need to work over discrete posets, however, many interesting cases arise from more general $P$. 
In the next example, the density of $P = \R$ allows us to choose assignments where the loss  $L(\phi,\psi)$ is not finite.

\begin{example}
Consider $\T_\e(m) =  m + \e$  from Example \ref{example_1}. 
Define the merge tress $F$ and $G$ in $\Set^{[0,\infty)}$  by
\[
\adjustbox{scale=1,center}{
\begin{tikzcd}[cramped,row sep = tiny, column sep = tiny]
	{F:} & {a_0} & {a_1} & {a_2} & {a_3} & \cdots & {a_t} \ar[r] &\cdots 
    & {F(t) = \{a_t\}}  \\
	{G:} & {b^0_0} & {b^0_1} & {b^0_2} & {b^0_3} & \cdots & {b^0_t} \ar[r]& \cdots 
    & {G(t) = \{b^0_t\} \cup \{b^j_t\}_{j=\floor{t}+1}^{\infty}} \cong \Z_{\geq 0} \\
	& {b^1_0} \\
	& {b^2_0} & {b^2_1} \\
	& {b_0^3} & {b_1^3} & {b_2^3}\\
    & \vdots & \vdots & \vdots 
	\arrow[from=1-2, to=1-3]
	\arrow[from=1-3, to=1-4]
	\arrow[from=1-4, to=1-5]
	\arrow[from=1-5, to=1-6]
	\arrow[from=1-6, to=1-7]
	\arrow[from=2-2, to=2-3]
	\arrow[from=2-3, to=2-4]
	\arrow[from=2-4, to=2-5]
	\arrow[from=2-5, to=2-6]
	\arrow[from=2-6, to=2-7]
	\arrow[from=3-2, to=2-3]
	\arrow[from=4-2, to=4-3]
	\arrow[from=4-3, to=2-4]
	\arrow[from=5-2, to=5-3]
	\arrow[from=5-3, to=5-4]
	\arrow[from=5-4, to=2-5]
\end{tikzcd}
}
\]
Consider the $\T_{\frac{1}{2}}$-assignment $\phi$  and $\psi$ given by
\begin{equation*}
\phi_t(a_t) = 
\begin{cases}
    b^0_{t+\frac{1}{2}} & \text {if} \quad t =0 \text{ or } t \geq 1\\
    b^n_{t+\frac{1}{2}} & \text {if} \quad \dfrac{1}{2^{n+1}} \leq t <  \dfrac{1}{2^{n}}
\end{cases}
\qquad \text{and} \qquad 
\psi_t(\bullet_t) = a_{t+\frac{1}{2}}.
\end{equation*}
We can check that for $\frac{1}{2^{n+1}} \leq t <  \frac{1}{2^{n}}$ and $n > 0$
\begin{equation*}
    \Lpl^{0,t}(\phi) 
    = d_{t + \frac{1}{2}}^G \left(b^0_{t + \tfrac{1}{2}},b^n_{{t + \tfrac{1}{2}}}\right) 
    = n - 1 + (1 - (t + \tfrac{1}{2})) = n - t - \tfrac{1}{2}.
\end{equation*}     
Since $G(0) \cong \N$, this means that there is a pair of elements with distance at least $N$ for any $N$, so  $L(\phi,\psi) = \infty$.
\end{example}

Intuitively, the assignment fails because $\phi$ is not regular when it gets closer to 0, a critical value of $G$. 
Hence, let us define assignments that are regular with respect to the functor.

\begin{definition}
Let $F,G \in C^P$ be given for a complete linear order $P$. 
A $\T_\e$-assignment $(\phi,\psi)$ is \textit{constructible} if for all $p \leq q$ in a constant interval $I$  of $F$ we have that $\Lpl^{p,q}(\phi) = 0$; 
and for all $p \leq q$ in $J$ a constant interval of $G$ we have $\Lpr^{p,q}(\psi) = 0$. 
\end{definition}

Let $F,G \in C^P$ be tame modules (Defn.~\ref{def:tameModule}) with critical values for $F$ and $G$ given respectively by 
$c_1 < \cdots < c_m  $ 
and $d_1 < \cdots < d_{n} $.
Set $c_0 = d_0 = -\infty$ and $c_{m+1} = d_{n+1} = \infty$ to represent respectively the infimum and supremum elements of $P$. 
The goal is to split the constant intervals $(c_i,c_{i+1})$ into intervals that map via $\T_\e$ into constant intervals of $G$, and then to use the $\phi$ and $\psi$ maps defined for representatives from the first piece of this decomposition to extend to values for the rest of the interval. 

For $i \in \{0,\ldots,m\}$, consider the (possibly empty) intervals 
\begin{equation*}
\begin{matrix}
J_{ij}  = (c_i , c_{i+1}) \cap \{ t \in P \mid  \T_\e t = d_{j} \} 
    & j \in \{1,\ldots,n \}\\
J'_{ij} = (c_i , c_{i+1}) \cap \{ t \in P \mid  \T_\e t \in (d_{j} , d_{j + 1}) \} 
    & j \in \{0,\ldots,n \}.
\end{matrix}
\end{equation*}
In the example of Figure \ref{fig_line_critical_points} with $i=2$, we have broken up the interval $(c_2,c_3)$ into three intervals: $J_{2,2} = (c_2,c_2+1)$, $J'_{2,2} = [c_2+1,c_2+3)$, and $J'_{2,3} = [c_2+3, c_3)$.  

Choose one element from each non-empty $J_{ij}$ or $J'_{ij}$, and write these as $\{\alpha_{i,k} \mid k \in 1,\cdots,k_i\}$.
Again following the example of Figure \ref{fig_line_critical_points}, we have 
$\alpha_{2,1} \in J_{2,2}$, 
$\alpha_{2,2} \in J_{2,2}'$, and
$\alpha_{2,3} \in J_{2,3}'$.

To simplify notation, we reindex all the non-empty intervals and write $I_{i,k}$ for the interval (whether of the form $J_{i,\cdot}$ or $J'_{i,\cdot}$) containing $\alpha_{i,k}$.
Note that $I_{i,k} \cap I_{i,k'} = \emptyset$ if $k \neq k'$ and $(c_i, c_{i+1}) = \bigcup_{k=1}^{k_i} I_{i,k}$.  
Moreover, by Lemma \ref{lemma_critical}, the image of $I_{i,k}$ under $\T_\e$ is a constant interval of $G$. 
Hence, we have divided $(c_i, c_{i+1})$ in intervals that map to constant intervals.
Let $R_F$ be the points of $P$ which are either chosen representatives or critical points $c_i$, specifically 
\[
R_F = 
    \big\{ \alpha_{i,k} \mid i \in \{0,\cdots,m\}, k \in \{1,\cdots,k_i \}\big\} 
    \cup 
    \big\{c_i \mid i  \in \{1,\cdots,m\}\big\}.
\]

In a similar way, we can break up constant intervals of $G$ of the form $(d_j, d_{j+1})$ into portions which map to constant intervals of $F$. 
We then choose representatives $\beta_{j,\ell}$ of these sub intervals, and define $R_G$ to be 
\[
R_G = 
    \big\{ \beta_{j,\ell} \mid j \in \{0,\cdots,n\}, \ell \in \{1,\cdots,\ell_j \}\big\} 
    \cup 
    \big\{d_j \mid j  \in \{1,\cdots,n\}\big\}.
\]
We will often use just the first representative in each interval, so we write $\alpha_i = \alpha_{i,1}$ and $\beta_j = \beta_{j,1}$, and then write the relevant subsets of $R_F$ and $R_G$ as 
\begin{align*}
    B_F &= 
    \big\{ \alpha_{i} \mid i \in \{0,\cdots,m\}\big\} 
    \cup 
    \big\{c_i \mid i  \in \{1,\cdots,m\}\big\}, \qquad \text{and}\\
    B_G &= 
    \big\{ \beta_{j} \mid j \in \{0,\cdots,n\}\big\} 
    \cup 
    \big\{d_j \mid j  \in \{1,\cdots,n\}\big\}.
\end{align*}

\begin{figure}
    \centering
    \includegraphics[width=\linewidth]{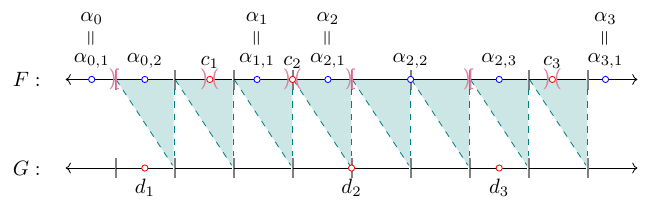}
    \caption{Example of the construction of the intervals $I_{i,k}$ for the map $\T_1: \R \to \R$ given by $\T_1(m) = \floor{m} + 1$ and the functors $F: \R \to C$ with critical set $\{c_1,c_2,c_3\}$ and $G: \R \to C$ with critical set $\{d_1,d_2,d_3\}$. For each interval we have $\alpha_{i,k} = I_{i,k}$.
    }
    \label{fig_line_critical_points}
\end{figure}

\begin{theorem}
\label{thm:ControlledLossLine}
  Let $F,G \in \C^P$ be tame and $(\phi,\psi)$  be a given constructible, Archimedean $\T_\e$-assignment between them. 
  Then, the loss is given by 
    \begin{align*}
        L(\phi,\psi) 
        = 
        \max \Bigg\{  
        &\max_{s \prec t \; \in B_F} \{\Lpl^{s,t}(\phi) + \delta_t\},  
        \max_{s \in R_F} \{\Ltd^s(\phi,\psi ) + \delta_s \}, \\
        &\max_{s \prec t \; \in B_G} \{\Lpr^{s,t}(\psi) +  \delta_t'\} , 
        \max_{s \in R_G } \{ \Ltu^s(\phi,\psi ) + \delta_s'\} 
        \Bigg\}
    \end{align*}
where $B_F \subset R_F$ and $B_G \subset R_G$ are defined as above,  $\delta_t$  and $\delta_s$ are given by 
\begin{align*}
\delta_t &= 
\begin{cases}
   0 &  t =c_i \\
   \sup\limits_{t' \; \in \; (c_i,c_{i+1})} \inf \{ \delta\geq 0 \mid \Lpl^{c_i,t'} (\phi) < \Lpl^{c_i,\alpha_i} (\phi) + \delta \}  & t = \alpha_{i} 
\end{cases} \\
\delta_s &= 
\begin{cases}
   0 &  s =c_i \\
   \sup\limits_{s' \; \in \; I_{i,k}} \inf \{ \delta\geq 0 \mid  \Ltd^{s'} < \Ltd^{\alpha_{i,k}} (\phi,\psi) + \delta \}  & s = \alpha_{i,k}
\end{cases} 
\end{align*}
and $\delta_t'$ and $\delta_s'$ are defined analogously using $\beta_{ij}$. 
 \end{theorem}

This theorem allows us to reduce the number of diagrams we need to consider to compute the overall loss of the assignment. For instance, instead of calculating $\Lpl^{p,q}(\phi)$ for $\sim |P|^2 $ number of pairs of the form $p < q$ in $P$ we only need to  calculate $ \Lpl^{s,t}(\phi)$ for  $|B| - 1 =  2m $  pairs $s \prec t$ in $B$. Similarly, instead of calculating $\Ltd^{p}(\phi,\psi)$ for all  elements $p$ in  $P$ we only need to  calculate $ \Ltd^{s}(\phi,\psi)$ for $s \in R_F \subset P$, a finite subset of $P$.

We first need a simple lemma to simplify the proof.  

\begin{lemma}
\label{lem:SameLossInInterval}
     Take $F,G \in \C^P$ and $(\phi,\psi)$ a  constructible, Archimidean $\T_\e$-assignment between them. 
     Let $(c,c')$ to be a constant interval of $F$. 
     Then, for all $s<t \in (c,c')$, 
    \begin{equation*}
     \Lpl^{s,c'}(\phi) = \Lpl^{t,c'}(\phi)
    \end{equation*}
     and there is a $\delta >0$ such that 
  \[
   \Lpl^{c,t}(\phi) \leq \Lpl^{c,s}(\phi) \leq \Lpl^{c,t}(\phi) + \delta.
  \]
\end{lemma}
\begin{proof}
    Take $s<t \in (c,c')$. 
    For the first statement, since $F[s \leq t]$ is isomorphism, using Lemma \ref{lemma_lll} statements 1) and 2) we have that
     \begin{align*}
        \Lpl^{s,c'}(\phi) &\leq \max \left\{ \Lpl^{s,t}(\phi) , \Lpl^{t,c'}(\phi) \right\} = \Lpl^{t,c'}(\phi) 
        \leq \max \left\{ \Lpl^{s,t}(\phi) , \Lpl^{s,c'}(\phi) \right\} = \Lpl^{s,c'}(\phi).
    \end{align*}
    The two equalities arise because $(\phi,\psi)$ is a  constructible $\T_\e$-assignment and thus  $\Lpl^{s,t}(\phi) = 0$.
    Similarly for the second statement, by Lemma \ref{lemma_lll} statements 1) and 3)  there is $\delta$ such that
    \begin{align*}
        \Lpl^{c,t}(\phi) &\leq \max \left\{ \Lpl^{c,s}(\phi) , \Lpl^{s,t}(\phi) \right\} = \Lpl^{c,s}(\phi) 
        \leq \max \left\{ \Lpl^{c,t}(\phi) , \Lpl^{s,t}(\phi) \right\} + \delta = \Lpl^{c,t}(\phi) + \delta
    \end{align*}
 completing the proof.
\end{proof}
 
 \begin{proof}[Proof of Thm.~\ref{thm:ControlledLossLine}]
 First consider the parallelogram terms. 
 We show that $\Lpl^{s' , t'}(\phi)\leq  \max_{s \prec t \; \in B_F} \{\Lpl^{s,t}(\phi) + \delta_t\}$ for all $s'<t' \in P$, which implies that this parallelogram term upper bounds the similar version in the original loss definition. 
For all $s' \leq t'$ in $(c_i,c_{i+1})$, by the assumption that $(\phi,\psi)$ is constructible we have by definition that   $\Lpl^{s',t'}(\phi) = 0$ so this case is done. 
Next, suppose $s'<t'$ in $P$ and there is at least one critical value of $F$  between them. 
Specifically, let $\{c_i, ..., c_j\}$ be the critical values between $s'$ and $t'$. i.e.~$s'<c_i< \cdots < c_j < t'$.
By induction and Lemma \ref{lemma_lll}, statement 1), we have  
\begin{equation*}
    \Lpl^{s',t'}(\phi) \leq 
    \max
    \left\{ 
    \Lpl^{x,y}(\phi) \mid 
    (x,y) \in \{ (s',c_i), (c_j,t')\} 
    \cup \{ (c_k,\alpha_k) \}_{k=i}^{j-1}
    \cup \{ (\alpha_k, c_{k+1}) \}_{k=i}^{j-1}
    \right\}.
\end{equation*}
By Lem.~\ref{lem:SameLossInInterval}, $\Lpl^{s',c_{i}}(\phi) = \Lpl^{{\alpha_{i-1}},c_{i}}(\phi)$ and there is $\delta_{t'}$ such that 
$\Lpl^{c_j,t'}(\phi) \leq \Lpl^{c_j,{\alpha_j}}(\phi) + \delta_{t'}$. So, 
\begin{equation*}
    \Lpl^{s',t'}(\phi) \leq 
    \max
    \left\{ 
    \Lpl^{x,y}(\phi) + \delta_y \mid 
    (x,y) \in  \{ (c_k,\alpha_k) \}_{k=i}^{j}
    \cup \{ (\alpha_k, c_{k+1}) \}_{k=i-1}^{j-1}
    \right\}
\end{equation*}
where $\delta_y = \delta_{t'}$ if $y = \alpha_j$ and $\delta_y = 0$ otherwise. 
It follows that 
\begin{equation*}
    \Lpl^{s',t'}(\phi) \leq  \max_{s \prec t \; \in B_F} \{\Lpl^{s,t}(\phi) + \delta_t\}.
\end{equation*}
A symmetric argument considering the critical values of $G$ and $\beta_j$ gives that 
$$
\Lpr^{s',t'}(\psi) \leq 
\max_{s \prec t \; \in B_G} \{\Lpr^{s,t}(\psi) + \delta_t'\}.
$$

For the triangle terms we will show that for any $s' \in P$ we have  $\Ltd^{s'}(\phi,\psi) \leq  \max_{s \in R_F} \{\Ltd^s(\phi,\psi ) + \delta_s \}$ which implies that this triangle term upper bounds the similar version in the original loss definition. Recall that 
\[ 
    P = \bigcup_{i=1}^m \{ c_i\} \cup \bigcup_{i=0}^m (c_i,c_{i+1}) = \bigcup_{i=1}^m \{ c_i\} \cup \bigcup_{i=0}^m \bigcup_{k=1}^{k_i} I_{i,k}.
\]
Then, for $s' \in P$, the element $s'$ is either a critical point of $F$ or $s' \in I_{i,k}$ for some interval $I_{i,k}$.  
If $s'$ is a critical point, it is in the set $R_F$ so we are done. 
In the case where $s' \in I_{i,k}$, by the definition of $I_{i,k}$ and the fact that $\T_\e$ is a constructible assignment we have that for any $x<y$ in $I_{i,k}$, $\Lpl^{x, y}(\phi) = 0$ and $\Lpr^{x, y}(\psi) = 0$. 
Given that $\T_\e$ is an Archimedean assignment, by Lemma \ref{lemma_ll_triang} (2) there is $\delta > 0$ such that 
\[
 \Ltd^{y}(\phi,\psi) \leq \Ltd^{x}(\phi,\psi) < \Ltd^{y}(\phi,\psi) + \delta.
\]
Given that both $s'$ and $\alpha_{i,k}$ are in $I_{i,k}$, this implies there is $\delta \geq 0$ such that
\[
 \Ltd^{s'}(\phi,\psi) <  \Ltd^{\alpha_{i,k}}(\phi,\psi) + \delta.
\]
Therefore, for any $s' \in P$,
\[
\Ltd^{s'}(\phi,\psi) \leq  \max_{s \in R_F} \{\Ltd^s(\phi,\psi ) + \delta_s \}.
\]
In the same way, using the critical values of $G$ and $\beta_{j,\ell}$ gives that  
\[
\Ltu^{s'}(\phi,\psi) \leq  \max_{s \in R_G} \{\Ltu^s(\phi,\psi ) + \delta_s' \}.
\]
\end{proof}

\begin{remark}
The intervals $I_{i,k}$ are of the form $\lfloor \inf I_{i,k} , \inf I_{i,k+1} \rceil$ or $\lfloor \inf I_{i,k} , c_{i+1} \rceil$ where the symbols $\lfloor$ and $\rceil$ denote that the lower and upper endpoints are either closed or open according to whether they are included in  $I_{i,k}$, respectively. Moreover, if we choose $\alpha_{i,k} = \min I_{i,k}$ when it exists we obtain that $\delta_{\alpha_{i,k}} = 0$.    
\end{remark}
  
\paragraph{Creating a constructible assignment }
Let us create a constructible $\T_\e$-assignment $(\phi,\psi)$ between $F$ and $G$ for given choices of maps on the sets $B_F$ and $B_G$. 
Specifically, assume we have chosen a collection of maps $\phi_r: F(r) \to G\T_\e (r)$ for $r \in B_F$. 
Let us extend the maps to an unnatural transformation $\phi: F \rightdsar G\T_\e$ by setting
\begin{equation*}
    \phi_t =
        \begin{cases}
             \phi_t &  \text{if } t \in B_F\\
             (G[\T_\e t \leq \T_\e \alpha_i])^{-1} \circ  \phi_{\alpha_i} \circ F[t \leq \alpha_i ] &  \text{if } t \in (c_i,\alpha_{i})\\
             G[\T_\e \alpha_i \leq \T_\e t ] \circ \phi_{\alpha_i} \circ (F[\alpha_i \leq t])^{-1} &  \text{if } t \in (\alpha_i,c_{i+1}).
        \end{cases}
\end{equation*}
Symmetrically, construct $\psi: G \rightdsar F\T_\e$ from a collection of maps $\psi_r: G(r) \to F\T_\e (r)$ for $r \in B_G$ and extending it for all values of $t \in P$.
Note that we are using the fact that the intervals $(c_i, c_{i+1})$ and $(d_j, d_{j+1})$ are constant for $F$ and $G$ respectively in order for these maps to be well defined.

Taking $(\phi, \psi)$ we obtain a $\T_\e$ assignment of $F$ and $G$. For this assignment we obtain that if $t \in (c_i,\alpha_{i})$ 
\begin{equation}
    G[\T_\e t \leq \T_\e \alpha_i] \circ \phi_{t} = G[\T_\e t \leq \T_\e \alpha_i] \circ (G[\T_\e t \leq \T_\e \alpha_i])^{-1} \circ  \phi_{\alpha_i} \circ F[t \leq \alpha_i] = \phi_{\alpha_i} \circ F[t \leq \alpha_i].
    \label{eq_before_a}
\end{equation}

Equivalently, if $t \in (\alpha_i,c_{i+1})$
\begin{equation}
    \phi_t \circ F[\alpha_i \leq t] = G[\T_\e \alpha_i \leq \T_\e t ] \circ \phi_{\alpha_i} \circ (F[\alpha_i \leq t])^{-1} \circ F[\alpha_i \leq t] = G[\T_\e \alpha_i \leq \T_\e t ] \circ \phi_{\alpha_i}.
    \label{eq_after_a}
\end{equation}

\begin{lemma}
The above defined $(\phi,\psi)$ is a constructible $\T_\e$-assignment.
\end{lemma}

\begin{proof}
 Note that for any $I$ which is a constant interval of $F$, $I \subseteq (c_i,c_{i+1})$ for some $i$. 
We need to show that for  $s < t$ in $I$, $\Lpl^{s,t}(\phi) = 0$.
 We have three cases to show that $ G[\T_\e s \leq  \T_\e t] \circ \phi_{s} = \phi_t \circ F[s \leq t] $.

\begin{itemize}
    \item If $s < t \leq \alpha_i$:
    \begin{align*}
        G[\T_\e t \leq \T_\e \alpha_i] \circ G[\T_\e s \leq  \T_\e t] \circ \phi_{s} & = G[\T_\e s \leq  \T_\e \alpha_i ] \circ \phi_{s} & \\
        & = \phi_{\alpha_i} \circ F[s \leq \alpha_i] & \text{by Eqn.} ~\eqref{eq_before_a} \\
        & = \phi_{\alpha_i} \circ F[t \leq \alpha_i] \circ F[s \leq t] & \\
        & = G[\T_\e t \leq \T_\e \alpha_i ] \circ \phi_t \circ F[s \leq t] & \text{by Eqn.} ~\eqref{eq_before_a}
    \end{align*}
    Since  $G[\T_\e t, \T_\e \alpha_i ]$ is isomorphism, $G[\T_\e t, \T_\e \alpha_i ]$ can be canceled from the left to obtain the result. 
    
    \item If $s < \alpha_i \leq t$.
    \begin{align*}
        G[\T_\e s \leq  \T_\e t] \circ \phi_{s} &= G[\T_\e s \leq \T_\e \alpha_i] \circ G[\T_\e \alpha_i, \T_\e t ] \circ \phi_s & \\
        &= G[\T_\e s \leq \T_\e \alpha_i] \circ \phi_{\alpha_i} \circ F[s \leq \alpha_i] & \text{Eqn.} ~\eqref{eq_before_a}\\
        &=\phi_t \circ F[\alpha_i, t]\circ F[s \leq \alpha_i] & \text{Eqn.} ~\eqref{eq_after_a} \\
        &= \phi_t \circ F[s \leq t] &
    \end{align*}
    \item If $\mathbf{\alpha_i \leq s < t}$
    \begin{align*}
        G[\T_\e s \leq  \T_\e t] \circ \phi_s \circ F[\alpha_i \leq s]  &= G[\T_\e s \leq  \T_\e t] \circ G[\T_\e \alpha_i \leq \T_\e s ] \circ \phi_{\alpha_i} &  \text{Eqn.} ~\eqref{eq_after_a}\\
        &=G[\T_\e \alpha_i \leq \T_\e t ] \circ \phi_{\alpha_i} &  \\
        &= \phi_t \circ F[\alpha_i \leq t]  & \text{Eqn.} ~\eqref{eq_after_a} \\
        &= \phi_t \circ F[s \leq t] \circ F[\alpha_i \leq s] &    
    \end{align*}
    
     Since $F[\alpha_i \leq s]$ is isomorphism, $F[\alpha_i \leq s]$  can be canceled from the right to obtain the result. 
\end{itemize}
\end{proof}

\subsection{Optimization over products of complete linear orders}

Let us generalize the previous result for modules over $P^k$, the $k$-th product of the linear poset $P$. Take $F,G$ tame modules of $C^{P^k}$ with a $\T_\e$-assignment between them that is constructible and line-preserving. 
The generalization has the following steps:
\begin{enumerate}
    \item Define the constant intervals of $P^k$ with respect to $F,G$.
    \item Subdivide each constant interval of $F$ into intervals that map into constant intervals of $G$ via $\T_\e$ ( and viceversa).
    \item Define a poset of representatives of the intervals.
    \item Prove the loss is bounded by the loss of the set of representatives.
\end{enumerate}

\subsubsection{Defining constant intervals on $P^k$}

Let $F$,$G$ in $C^{P^k}$ be cubical tame modules (Def.~\ref{def:cubicalTame}) and let 
$x^i_2 < x^i_4< \cdots < x^i_{2m_i}$ 
be the critical $i$-coordinates of $F$. 
Similarly, let $y^i_2 < \cdots < y^i_{2n_i}  $ be the critical $i$-coordinates of $G$. 
Set $x^i_0 = y^i_0 =-\infty$ and $x_{2(m+1)} =  y_{2(n + 1)} =   \infty$ where $-\infty$ and $\infty$ represent respectively the infimum and supremum elements of $P$. 

First, we split $P^k$ into constant intervals of $F$. 
Consider the set of indices 
\begin{equation*}
\Lambda = \{ (\tuple{j}{k}) \mid \text{for }0 \leq j_i \leq 2m_i+1 \}. 
\end{equation*}
For each $I = (\tuple{j}{k})$ in $\Lambda$, define the cube $S_{I} =  X^1_{j_1} \times \cdots \times X^k_{j_k} $ where $X^i_{j} = \{x^i_{j}\} $ if $j$ is even and $X^i_{j} = (x^i_{j - 1},x^i_{j+1}) $ if $j$ is odd.
For example, in Figure \ref{fig_plane_critical_lines}, the cube $S_{33} = (x^1_2,x^1_4) \times (x^2_2,x^2_4)$, meanwhile $S_{32} = (x^1_2,x^1_4) \times \{x^2_2\}$.

\begin{figure}
    \centering
    \includegraphics[width=0.45\textwidth]{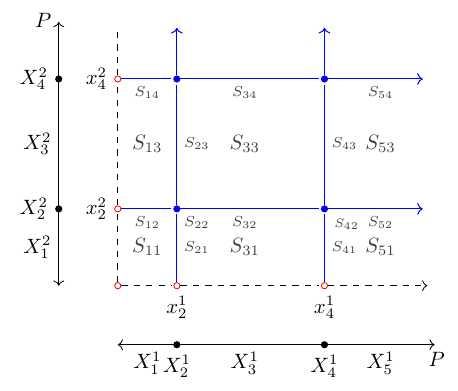}
    \hspace{1em}
    \includegraphics[width=0.45\textwidth]{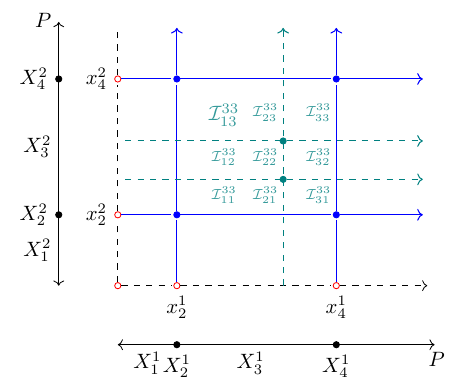}
    \caption{On the left: Example of the construction of the cubes $S_{(j_1,j_2)} = X^1_{j_1} \times X^2_{j_2}$  for the functors $F: P^2 \to C$ with critical coordinates $\{x^1_2,x^1_4\}$ and  $\{x^2_2,x^2_4\}$. On the right: We divide the cube $S_{33}$ into interval $\mathcal{I}_J^{33}$ that map into the  constant cube $R_J$ by the morphism $\T_\e$. For example,  $ p \in \mathcal{I}_{12}^{33}$ if and only if $p \in S_{33}$ and $\T_\e p \in R_{12}$. } 
    \label{fig_plane_critical_lines}
\end{figure}

\begin{theorem}
    The cubes $S_{I} \subseteq P^k$ are constant intervals of $F$.
\end{theorem}
    \begin{proof}

        Fix $I = (\tuple{j}{k})\in \Lambda$. 
        Since each $X^i_{j}$ is a convex and connected subset of $P$ we have that $S_{I}$ is a convex and connected subset of $P^k$, hence is an interval.
        We will show that $F[p \leq q]$ is an isomorphism for any $p \leq q \in S_I$ and thus $S_I$ is constant.
        If all $j_i$'s are even, $S_{I}$ is a point and $F$ is constant on a point so we are done. 
        Thus assume $I$ has at least one odd index and take $p < q \in S_I$.

        We will do induction on the number of differing entries of $p$ and $q$. 
        If $p$ and $q$ share all coordinates but one, we can write 
        $p = \ncoord{p}{k}$ and 
        $q = \anglebracket{p^1, \ldots, p^{i-1}, q^i ,p^{i+1}, \ldots,  p^k}$ 
        with $p^i < q^i$ and $j_i$ odd, as denoted below: 
        \[\begin{tikzcd}[cramped]
    	{\anglebracket{p^1, \ldots, p^{i-1}, q^i ,p^{i+1}, \ldots,  p^k}} \\
    	\\
    	{\ncoord{p}{k}}.
    	\arrow[from=3-1, to=1-1]
        \end{tikzcd}\]
        If $F[p \leq q]$ is not an isomorphism, then there is a critical value 
        $\anglebracket{p^1, \ldots, p^{i-1}, t ,p^{i+1}, \ldots,  p^k} $  
        of $F(\anglebracket{p^ 1, \ldots, p^{i-1}, \cdot ,p^{i+1}, \ldots,  p^k}): P \to C$ 
        with $x^i_{j_i - 1} < p_i \leq t \leq q_i < x^i_{j_i + 1}$.  
        However, the latter statement 
        contradicts  the exhaustive list of critical $i$-coordinates. 
        Hence $F[p \leq q]$ is isomorphism.

        Now assume $p$ and $q$ differ in  $n$ coordinates, and assume that any points differing in $n-1$ coordinates result in an isomorphism.
        Write  $p = \ncoord{p}{k} \leq  \ncoord{q}{k} = q$ in $S_I$ let $i$ be one of the coordinates that differ between the two. 
        Without loss of generality, assume $p^i < q^i$ and set   
        \[r = \anglebracket{p^1, \ldots, p^{i-1},q^i, p^{i+1}, \ldots, p^k}.\] 
        so we have the diagram
        \[\begin{tikzcd}[cramped]
    	{\anglebracket{p^1, \ldots, p^{i-1},q^i, p^{i+1}, \ldots, p^k}} && {\ncoord{q}{k}} \\
    	\\
    	{\ncoord{p}{k}}
    	\arrow[from=1-1, to=1-3]
    	\arrow[from=3-1, to=1-1]
    	\arrow[from=3-1, to=1-3]
        \end{tikzcd}\]
        By induction, $F[p \leq r]$  and $F[r \leq q]$ are isomorphisms. 
        
    Then, since $F[p \leq q] = F[r \leq q] \circ F[p \leq r] $ we know that $F[p \leq q]$ is an isomorphism. 
    Therefore, $F$ is constant on $S_{I}$.
    \end{proof}

We note that the indices in $\Lambda$ used to enumerate the $S_I$ are built to be compatible with the ordering of the $P^k$. 
We make this explicit in the following lemma.

\begin{lemma}
    \label{lemma:compatible_order}
    The product order on $\Lambda \subseteq \N^k$ used to index the sets $S_{I}$ is compatible with the order of $P^k$. Explicitly, the following statements hold:
    \begin{enumerate}[(a)]
        \item If $p \leq q \in P^k$ with $p \in S_{I}$ and $q \in S_{I'}$, then $I \leq_{\N^k} I'$. 
        \item If $I \leq_{\N^k} I'$ and  $q \in S_{I'}$, there is a $p \in S_{I}$  such that  $p \leq q$.
    \end{enumerate} 
\end{lemma}
\begin{proof}
    \textit{(a)} Let $I = (\tuple{j}{k})$ and $I' = (\tuple{j'}{k})$ be given elements in $\Lambda$. 
    On one hand assume that $p \leq q \in P^k$ with $p \in S_{I}$ and $q \in S_{I'}$. 
    Then, for each $1 \leq i \leq k$, $p^i \in X^i_{j_i}$ and $q^i \in  X^i_{j'_i} $ and $p^i \leq q^i$. 
    By construction of the $X_j$'s we can check that this implies that $j_i \leq j'_i$.

    \textit{(b)} Assume that $I \leq I'$ and choose $q \in S_I'$. 
    If $I  =I'$ we are done since $q$ is comparable to itself. 
    Hence, suppose that $I < I'$ and take $p \in S_I$. If $q < p$. 
    By part \textit{(a)}, $I' \leq I$, but this contradicts the assumption that $I < I'$. 
    So if $p$ and $q$ are comparable, then $p \leq q$ as desired.
    
        If $p$ and $q$ are not comparable, consider the non-empty sets 
        \mbox{$U = \{ i \mid p^i \leq q^i\}$}  and $V= \{ i \mid p^i > q^i\}$. 
        Define $r \in P^k$ with $r^i = p^i$ if $i \in U$ and $r^i = q^i $ if $i \in V$. 
        By construction we have that $r <p$ and $r<q$ as shown in the following diagram:
        \[\begin{tikzcd}[cramped]
           {\ncoord{q}{k}} &&  \\
        \\
        {\ncoord{r}{k}} && {\ncoord{p}{k}}.
        \arrow[from=3-1, to=1-1]
        \arrow[from=3-1, to=3-3]
        \end{tikzcd}\]
        
         Let $I'' \in \Lambda$ be such that  $r \in S_{I''}$; we will show that $I'' = I$, and thus $r \in S_{I}$ which completes the proof. 
        If $i \in U $ then $j_i'' = j_i$ because $X^i_{j_i''} \ni r^i = p^i \in X^i_{j_i}$. 
        Similarly, for $i \in V $ we have that $j_i'' = j_i'$ given that $X^i_{j_i''} \ni r^i = q^i \in X^i_{j_i' }$. 
        Now, given that $r < p$ and by part \textit{(a)}, this implies that $I'' \leq I$. 
        Hence, if $i \in V $,
        $j_i \leq j'_i = j_i'' \leq j_i$. 
        Therefore, $j_i'' = j_i$ if $i \in V$ as well, so $I'' = I$.      %
\end{proof}

Note that the collection $\{S_{I} \mid I \in \Lambda\}$ covers all of $P^k$.
We will call $\{S_{I} \mid I \in \Lambda\}$ the set of constant cubes of $F$. 
We can define similarly $\Lambda'$ for the indices of the critical values of $G$ and write $\{R_{J} \mid J \in \Lambda'\}$ for the set of constant cubes of $G$ of the form  $Y^1_{j_1} \times \cdots \times Y^k_{j_k}$.

\subsubsection{Subdividing into intervals that map to constant intervals}

Next, we will subdivide each $S_{I}$  into intervals that map into constant intervals of $G$ via $\T_\e$. 
For each $I \in \Lambda$ and $J \in \Lambda'$, define 
\[ \mathcal{I}_{J}^{I} = S_I \cap \{ p \in P^k \mid \T_\e  p \in R_J \}. \]
Analogously, we define 
$$ \mathcal{J}_{I}^{J} = R_J \cap \{ p \in P^k \mid \T_\e p \in S_I\}$$
for the subdivision of the constant intervals of $G$ mapping into constant intervals of $F$ via $\cT_\e$.

\begin{remark}
Because we assumed that $\T_\e$ is line-preserving, $\mathcal{I}_{J}^{I}$ can be written only in terms of terms of the form $\inf \mathcal{I}_{J'}^{I}$ as follows. 
Consider $U = \{ 1 \leq i \leq k \mid p^i \neq q^i \text{ for some } p,q \in  \mathcal{I}_{J}^{I} \}$. 
Let $I =(\tuple{j}{k})$ and $J = (\tuple{\hat{j}}{k})$. 
For each $i \in U$ take the index 
$$J_i = (\tuple{\hat{j}}{i-1}, \hat{j}_i + 1, \tuple[i+1]{\hat{j}}{k} ).$$ 
Then
\begin{equation*}
    \mathcal{I}_J^I = \left\{\begin{array}{cc}
         \inf \mathcal{I}_{J}^{I} & \text{ if } U = \emptyset  \\
        \displaystyle \bigtimes_{i \in U} \left\lfloor \: (\inf \mathcal{I}_{J}^{I})^i , (\inf \mathcal{I}_{J_i}^{I})^i \: \right\rceil&   \text{ if } U \neq \emptyset \text{ and } \inf \mathcal{I}_{J_i}^{I} \neq \emptyset \\
        \displaystyle \bigtimes_{i \in U} \left\lfloor \: (\inf \mathcal{I}_{J}^{I})^i , x^i_{j_i +1} \: \right\rceil&   \text{ if } U \neq \emptyset \text{ and } \inf \mathcal{I}_{J_i}^{I} = \emptyset
    \end{array}
    \right.
\end{equation*}

Where $(a)^i$ denotes the $i$-th coordinate of $a \in P^k$. The symbols $\lfloor$ and $\rceil$ denote that the lower and upper endpoints are closed or open according to whether they are included in  $\mathcal{I}_J^I$, respectively.
Hence, when  $\T_\e$ is line-preserving the intervals $\mathcal{I}_J^I$ and $\mathcal{J}_{I}^{J}$ are cubes.
\end{remark}

\begin{lemma}
   \label{lemma:interval_order} 
    Let $I \leq I' \in \Lambda$ and $J < J' \in \Lambda'$, and suppose that  
    $\mathcal{I}_{J}^{I}$ and $\mathcal{I}_{J'}^{I'}$ are non-empty. 
    Then, for each $q \in \mathcal{I}_{J'}^{I'}$ there is $p \in \mathcal{I}_{J}^{I}$ with $p \leq q$.
\end{lemma}
    \begin{proof}
    Assume that $J = (\tuple{j}{k}) < (\tuple{j'}{k}) =  J' $. 
    Let $q \in \mathcal{I'}_{J'}^{I}$ and $p \in \mathcal{I}_{J}^{I}$ be given. 
    If $q < p$, then $\cT_\e q \leq \cT_\e p$, but this contradicts the assumption that $J < J'$. So if $p$ and $q$ are comparable, then $p \leq q$ as desired.
    
    If $p$ and $q$ are not comparable, define $r$ using the sets $U$ and $V$ as in the proof of Lemma \ref{lemma:compatible_order}(b). 
    Since $p\in S_I$, $r \in S_I$ by the lemma.  
    Let $J'' \in \Lambda'$ be such that  $\T_\e r \in R_{J''}$, so by construction, $r \in \mathcal{I}_{J''}^I$.
  
    Since $\T_\e$ is line-preserving and by using Lemma \ref{lemma:line-pres}, for each $1 \leq i \leq k$ there are translations $\T_\e^i\colon P \to P$ such that 
        \[
        \T_\e r = (\T_\e^1(r^1), \ldots, \T_\e^k(r^k) ).
        \]
    This implies that if $i \in U $, $(\T_\e(r))^i = (\T_\e(p))^i$; and if $i \in V $, $(\T_\e(r))^i = (\T_\e(q))^i$. 
    Hence, by the same argument from Lemma \ref{lemma:compatible_order}, we have that $\T_\e r \in R_J$.  
    Thus $r \in \mathcal{I}_{J}^{I}$ completing the proof.
    \end{proof}

\subsubsection{Choosing representatives for each interval}

In Section \ref{ssec:Linear_orders_opt} where we studied modules over $P = P^1$, we could choose a representative $\alpha_{i,k}$ in the constant intervals of $F$ that map to constant intervals of $G$ since all elements in these intervals are comparable to their representative. 
However, when we have modules over $P^k$ for $k\geq 2$, this is no longer the case since we can find elements in a cube $I$ that are not comparable to each other. 
For example, consider $I= (0,1) \times (0,1) \in \R^2$. 
For all $\anglebracket{x,y} \in I$ there is $\e > 0$ such that $\anglebracket{x-\e,y + \e } \in I$ and $\anglebracket{x-\e,y + \e }$ is not comparable to $\anglebracket{x,y}$. 
Therefore, we need a different strategy for finding representatives in $P^k$.

For a given flow $\T$ over $P^k$, the next definition gives a poset $\hat{P}$ by constructing symbols $\alpha^{I}_{J}$ that will be used to represent each $\mathcal{I}_{J}^{I}$. 
Specifically, the elements of $\hat{P}$ will be written as  $\mathcal{S}\alpha^{I}_{J}$ for each  $\mathcal{S} \in \Flows_{\T}$  (Def.~\ref{def:flowsCategory}). 
While the actual  evaluation of  $\mathcal{S}$ on $\alpha^{I}_{J}$ is not defined since $\T$ is not defined over $\hat{P}$, only over $P^k$, we will use it to symbolically represent elements that get mapped to the inverse limit of the constant intervals.

\begin{definition}
Let 
$ \hat{P} = \{ \mathcal{S} \alpha^{I}_{J} \mid \mathcal{I}^I_J \neq \emptyset \text{ and } \mathcal{S} \in  \Flows_{\T}\} $. 
Define the relation $\leq$ in  $\hat{P} $ by the following:
    \begin{itemize}
        \item If $I \leq I'$ and $J \leq J'$, for all  $\mathcal{S} \in \Flows_{\T}$, 
        then $\mathcal{S} \alpha^{I}_{J} \leq \mathcal{S} \alpha^{I'}_{J'} $. 
        \item For all $\mathcal{S} \leq \mathcal{S}'$ in $\Flows_{\T}$, $\mathcal{S} \alpha^{I}_{J} \leq \mathcal{S}' \alpha^{I}_{J} $.
    \end{itemize}
\end{definition}
Note that for ease of notation, we will write the element $\alpha^{I}_{J} = \Id \; \alpha^{I}_{J} $. 
Also, since $ \Id \in \cS$, the first assumption implies that if $I \leq I'$ and $J \leq J'$, then $\alpha^{I}_{J} \leq \alpha^{I'}_{J'}$.

With this construction, we define the map $\hat{F}: \hat{P} \to C$ by using the inverse limit, explicitly
\begin{equation*}
     \hat{F}(\mathcal{S} \alpha^{I}_{J}) = \varprojlim (F(\mathcal{S} p) \mid p \in \mathcal{I}_{J}^{I}).
\end{equation*}

If the index $I$ is clear by context or not necessary we write $\alpha_J$ instead of $\alpha^I_J$.

We next show that this map is indeed functorial.

\begin{theorem}
\label{thm:hatF_is_functor}
    The map $\hat{F}: \hat{P} \to C$ is a functor.
\end{theorem}

\begin{proof}
First, we need to prove that for $\hat{p} \leq \hat{q}$ in $\hat{P}$, $\hat{F}(\hat{p}) \to \hat{F}(\hat{q})$ exists. 
Let us start by finding the map $F(\alpha_J^I) \to F(\alpha_{J'}^{I'})$ when $I \leq I'$ and $J \leq J'$. 
Recall that in Lemma \ref{lemma:interval_order}, if we take $q \in \mathcal{I}_{J'}^{I'}$ there is $ p \in   \mathcal{I}_J^{I}$ with $p \leq q$. 
Take $q \leq q' $ in $\mathcal{I}_{J'}^{I'}$, then $p \leq p' $ in $\mathcal{I}_{J}^{I}$.
Given that $\hat{F}(\alpha_{J}^I)$ is the inverse limit and following the diagram on the left of Fig.~\ref{fig:dgms_defining_hatF} we have that
          \begin{equation*}
            F[p' \leq q'] \circ (F[p \leq p'] \circ F[\alpha_{J} \leq p]) = F[p' \leq q'] \circ (F[\alpha_{J} \leq p']). 
          \end{equation*}
          By the Thin Lemma we know that  
          $
          F[p' \leq q'] \circ F[p \leq p'] = F[q \leq q'] \circ F[p \leq q]$. 
          Plugging this equality on the left side of the  equation yields 
        \begin{equation*}
            F[q \leq q'] \circ F[p \leq q] \circ F[\alpha_{J} \leq p] = F[p' \leq q'] \circ F[\alpha_{J} \leq p'].
        \end{equation*}
        Setting $\gamma_q = F[p \leq q] \circ F[\alpha_{J} \leq p] $, this is the same as writing 
         $F[q \leq q'] \circ \gamma_q = \gamma_q'$.

    \begin{figure}
    \centering
    \begin{minipage}{0.45\textwidth}
       \[
       \adjustbox{scale = .9, center}{
       \begin{tikzcd}[cramped]
        & {\hat{F}(\alpha_{J}^I)} \\
        & {\hat{F}(\alpha_{J'}^{I'})} \\
        {F(q)} && {F(q')} \\
        {F(p)} && {F(p')}
        \arrow[dashed, from=1-2, to=2-2]
        \arrow[from=1-2, to=3-1]
        \arrow[from=1-2, to=3-3]
        \arrow[from=2-2, to=4-1]
        \arrow[from=2-2, to=4-3]
        \arrow[dotted, from=3-1, to=3-3]
        \arrow[from=3-1, to=4-1]
        \arrow[from=3-3, to=4-3]
        \arrow[from=4-1, to=4-3]
        \end{tikzcd}
        }
        \]
    \end{minipage}
    \qquad
    \begin{minipage}{0.45\textwidth}
        \[
       \adjustbox{scale = .9, center}{
        \begin{tikzcd}[cramped]
            & {\hat{F}(\cS \alpha_{J})} \\
            & {\hat{F}(\cS'\alpha_{J})} \\
            {F(\cS p)} && {F(\cS p')} \\
            {F(\cS' p)} && {F(\cS' p')}
            \arrow[dashed, from=1-2, to=2-2]
            \arrow[from=1-2, to=3-1]
            \arrow[from=1-2, to=3-3]
            \arrow[from=2-2, to=4-1]
            \arrow[from=2-2, to=4-3]
            \arrow[dotted, from=3-1, to=3-3]
            \arrow[from=3-1, to=4-1]
            \arrow[from=3-3, to=4-3]
            \arrow[from=4-1, to=4-3]
            \end{tikzcd}
            }\]
    \end{minipage}
    \caption{Diagrams for defining $\hat F$ in the proof of Thm.~\ref{thm:hatF_is_functor}.}
    \label{fig:dgms_defining_hatF}
\end{figure}
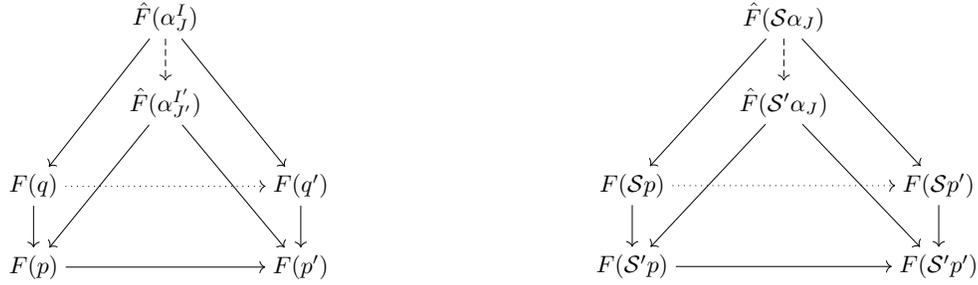

        Hence, by universal property of  $\hat{F}(\alpha_{J'}^{I'})$ we obtain a unique morphism $\hat{F}(\alpha_{J}^I) \to \hat{F}(\alpha_{J'}^{I'})$ that gives rise to $\hat{F}[\alpha_{J}^{I} \leq \alpha_{J'}^{I}]$. Moreover,  from this property we obtain that 
         \begin{equation}
            F[p \leq q] \circ F[\alpha_{J}^{I} \leq p] = F[\alpha_{J'}^{I'} \leq q] \circ \hat{F}[\alpha_{J}^{I} \leq \alpha_{J'}^{I'}] 
            \label{eq_F_commute}
        \end{equation}
        for any $p \leq q $ with in $p \in \mathcal{I}_{J}^{I} $ and  $q \in \mathcal{I}_{J'}^{I'} $.  By a similar argument, for any $\cS \in \Flows_\T$, we can find the morphisms $\hat{F}[\mathcal{S} \alpha_{J}^{I} \leq \mathcal{S}\alpha_{J'}^{I'}]$ .
       
        Similarly for all for all $\mathcal{S} \leq  \mathcal{S}'\in \Flows_{\T}$ we have the morphisms $\hat{F}[\mathcal{S} \alpha_{J} \leq \mathcal{S}'\alpha_{J}]$ by the universal property of $F(\mathcal{S'} \alpha_{J})$ as shown in the diagram on the right of Fig.~\ref{fig:dgms_defining_hatF}.
        Additionally, the uniqueness from the universal property ensures that $\hat{F}$ is compatible with the identity morphism and the associative property of the composition.
    \end{proof}

Symmetrically, we define the functor $\hat{G}:\hat{P} \to C$ given by
\begin{equation*}
     \hat{G}(\mathcal{S} \alpha^{I}_{J}) = \varprojlim (G(\mathcal{S} p) \mid p \in \mathcal{I}_{J}^{I})
\end{equation*}
with the canonical morphism $G[ \mathcal{S} \alpha^{I}_{J} \leq \mathcal{S} p ]:\hat{G}(\mathcal{S} \alpha^{I}_{J}) \to  G(\mathcal{S} p)$.

Note that for each $p \in \mathcal{I}_{J}^{I}$ we have a canonical morphism $\hat{F}(\mathcal{S} \alpha^{I}_{J}) \to F(\mathcal{S} p) $. 
We will abuse  notation and call this morphism $F[\mathcal{S} \alpha^{I}_{J} \leq \mathcal{S} p]$.  

Assume we are given $(\phi,\psi)$, a constructible $\T_\e$-assignment between $F$ and $G$, defined by $\phi: F \to G\T_\e$ and $\psi: F \to G\T_\e$. 
Our next goal is to "expand" these collection to $\phi_{\alpha_J}: \hat{F} (\alpha_J) \to \hat{G} (\T_\e \alpha_J) $ and 
$\psi_{\T_\e \alpha_J}: \hat{G} (\T_\e \alpha_J) \to \hat{F} (\T_\e \T_\e \alpha_J)$ for all $\alpha_{J}^I$ in $\hat{P}$.
 
Consider the diagram on the left of Figure \ref{fig_universal_diag} where $p \leq q \in \mathcal{I}_{J}^{I}$. 
We can verify that the triangle with corners $\hat F(\alpha_J)$, $F(p)$, and $F(q)$ commutes since $\hat F(\alpha_{J})$ is the inverse limit of $F$. 
By assumption, $F$ is constant on $ \mathcal{I}_{J}^{I} \subseteq S_{I}$ and thus the bottom square commutes; equivalently $\Lpl^{p,q} = 0$.
Thus the diagram 
\begin{equation*}
\begin{tikzcd}
&\hat F(\alpha_J) \ar[dl, "{\phi_p \circ F[\alpha_{J} \leq p]}"'] \ar[dr,"{\phi_q \circ F[\alpha_{J} \leq q]}"] \\
G\cT_\e (p) \ar[rr,"{G[\T_\e p \leq \T_\e q]}"'] && G\cT_\e(q)
\end{tikzcd}
\end{equation*}
commutes, making it a cone. 
But since $\hat G(\cT_\e \alpha_J)$ is a inverse limit, this means there is a unique map 
 \mbox{$\phi_{\alpha_{J}}\colon  \hat{F}(\alpha_j) \to \hat{G}(\T_\e \alpha_{J}) $} making the left diagram of Fig.~\ref{fig:dgms_defining_hatF} commute, i.e.~with the property that
\begin{equation}
    \phi_p \circ F[\alpha_{J} \leq p] = G[\T_\e \alpha_{J} \leq \T_\e p ] \circ \phi_{\alpha_{J}}.
    \label{eq:phi_commute}
\end{equation}

Symmetrically, we can define $\psi$ in such a way that the diagram on the right of Figure \ref{fig_universal_diag} commutes. 
Hence, the following equation holds: 
\begin{equation}
    \psi_{\T_\e p} \circ G[\T_\e \alpha_J \leq \T_\epsilon p] = F[\T_\e \T_\e\alpha_J \leq \T_\e \T_\epsilon p] \circ \psi_{\T_\e \alpha_J}
    \label{eq:psi_commute}
\end{equation}
In this way, for $\alpha_J^I \leq \alpha_{J'}^{I'}$ we can define the parallelogram diagram  
$\Parallelograml_\phi(\alpha_{J}^{I},\alpha_{J'}^{I})$ as well as the triangle diagram 
$\triangled_{\phi,\psi}(\alpha_{J}^I) $ in the same manner as in Sec.~\ref{ssec:lossfunction}.

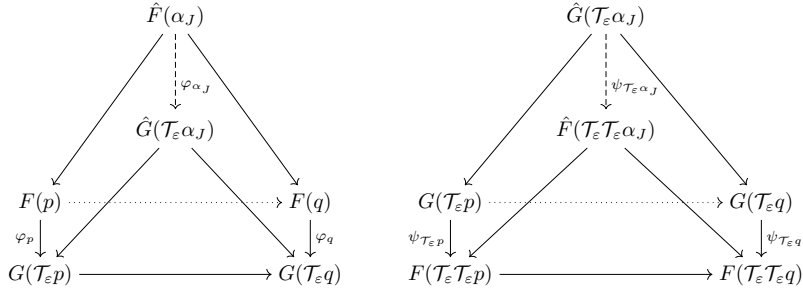
\begin{figure}
    \centering
    \adjustbox{scale=.8}{
    \begin{tikzcd}[cramped]
	& {\hat{F}(\alpha_{J})} \\
	\\
	& {\hat{G}(\T_\e \alpha_{J})} \\
	{F(p)} && {F(q)} \\
	{G(\T_\e p)} && {G(\T_\e q)}
	\arrow["{{\phi_{\alpha_{J}}}}"{pos=0.7}, dashed, from=1-2, to=3-2]
	\arrow[from=1-2, to=4-1]
	\arrow[from=1-2, to=4-3]
	\arrow[from=3-2, to=5-1]
	\arrow[from=3-2, to=5-3]
	\arrow[dotted, from=4-1, to=4-3]
	\arrow["\phi_p"',from=4-1, to=5-1]
	\arrow["\phi_q", from=4-3, to=5-3]
	\arrow[from=5-1, to=5-3]
\end{tikzcd} \hspace{24pt}
\begin{tikzcd}[cramped]
	& {\hat{G}( \T_\e \alpha_{J})} \\
	\\
	& {\hat{F}(\T_\e \T_\e \alpha_{J})} \\
	{G(\T_\e p)} && {G(\T_\e q)} \\
	{F(\T_\e \T_\e p)} && {F(\T_\e \T_\e q)}
	\arrow["{\psi_{\T_\e \alpha_{J}}}"{pos=0.7}, dashed, from=1-2, to=3-2]
	\arrow[from=1-2, to=4-1]
	\arrow[from=1-2, to=4-3]
	\arrow[from=3-2, to=5-1]
	\arrow[from=3-2, to=5-3]
	\arrow[dotted, from=4-1, to=4-3]
	\arrow["\psi_{\T_\e p}"',from=4-1, to=5-1]
	\arrow["\psi_{\T_\e q}", from=4-3, to=5-3]
	\arrow[from=5-1, to=5-3]
\end{tikzcd}
}
    \caption{Diagrams that show the existence of the maps $\hat{F}(\alpha_{J}) \to \hat{G}(\T_\e \alpha_{J}) $ and $\hat{G}(\alpha_{J}) \to \hat{F}( \T_\e\alpha_{J})$ by the universal property.}
    \label{fig_universal_diag}
\end{figure}

Note that all this work has been done to define $\phi$ and $\psi$ on $\hat P$, i.e.~constant intervals of $F$ that map to constant intervals of $G$. 
However, we need all this information for constant intervals of $G$ that map to constant intervals of $F$. 
Thus, in a similar fashion, we will define 
\[
\tilde{P} = \{ \mathcal{S} {\beta}^{I}_{J} \mid \mathcal{J}_I^J \text{ and } \mathcal{S} \in \Flows_{\T} \} 
\] 
on symbols $\beta_{J}^I$ with similar order relations to $\hat P$. 
Further, define $\tilde{G}\colon  \tilde{P} \to C$ and $\tilde{F}\colon  \tilde{P} \to C$ by
\begin{equation*}
    \tilde{G}(\cS \beta_{I}^J) = 
    \lim_{\longleftarrow} \left(G(\cS p) \mid p \in \mathcal{J}_I^ J\right) 
    \qquad  \text{and} \qquad
    \tilde{F}(\cS \beta_{I}^J) = \lim_{\longleftarrow} \left(F(\cS p) \mid p \in \mathcal{J}_I^ J\right).
\end{equation*}
Using the same arguments, both $\tilde{G}$ and $\tilde{F}$ are functors. We also have the maps \mbox{$\psi_{\beta_{I}}\colon  \tilde{G}(\beta_{I}) \to \tilde{F}(\T_\e \beta_{I})$} and \mbox{$\psi_{\T_\e \beta_{I}}\colon  \tilde{F}(\T_\e \beta_{I}) \to \tilde{G}(\T_\e \T_\e \beta_{I}) $}.
Putting all these together, we have diagrams as before: the parallelogram diagram $\Parallelogramr_\psi(\beta_{I}^{J},\beta_{I'}^{J'})$ and the triangle diagram $\triangleu_{\phi,\psi}(\beta_{I}^J) $. 
\subsubsection{Loss function of the representatives set}

Notice that even though we have not defined a flow $\T$ on $\hat{P}$, we can still define a merging distance (Defn.~\ref{def_merging})
$d^{\hat{F}}_{\hat{p}}$ 
on $\hat{F}(\hat{p})$  
for $\hat{p} \in \hat{P}$ given by 
\begin{equation*}
    d_{\hat p}^{\hat F}(a,b) = 
    \inf\{\e \geq 0 \mid \hat F[\hat p \leq T_{\e} \hat p](a) =  \hat F[\hat p \leq T_{\e} \hat p](b) \}.
\end{equation*}
Similarly we have the distance  $d^{\hat{G}}_{\hat{p}}$  on $\hat{G}(\hat{p})$. 
Hence, for $\alpha_J^I \leq \alpha_{J'}^{I'}$ we can define 
the loss $\Lpl^{\alpha_{J}^{I},\alpha_{J'}^{I'}}(\phi)$  
of the parallelogram diagram 
$\Parallelograml_\phi(\alpha_{J}^{I},\alpha_{J'}^{I})$  
as well as loss $\Ltd^{\alpha_{J}^I}(\phi,\psi )$ 
of the triangle diagram 
$\triangled_{\phi,\psi}(\alpha_{J}^I) $ in the same manner as Defn.~\ref{def:loss}.

\begin{lemma}
\label{lem:lossInequalitiesForHatP}
    For any $p \leq q$ with $p \in \mathcal{I}_J^I$ and $ q \in \mathcal{I}_{J'}^{I'}$ we have
    \begin{equation*}
         \Lpl^{p,q}(\phi) \leq \Lpl^{\alpha_{J}^I,\alpha_{J'}^{I'}}(\phi)
     \qquad \text{and} \qquad
        \Ltd^p(\phi,\psi) \leq \Ltd^{\alpha_{J}^I}(\phi,\psi).
    \end{equation*}
\end{lemma}
    \begin{proof}
        \begin{figure}
        \includegraphics[page=1,width=0.5\textwidth]{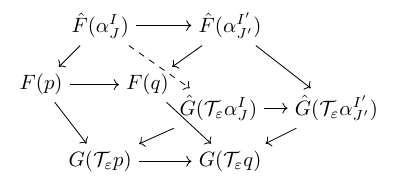}
        \hspace{24pt}
        \includegraphics[page=2, width=0.41\textwidth]{images/3d_parallelogram_diagram.pdf}

        \caption{Diagrams for the proof of Lem.~\ref{lem:lossInequalitiesForHatP}}
        \label{fig:DgmsForlossInequalitiesForHatP}
    \end{figure}
    Fix $p \leq q$ with $p \in \mathcal{I}_J^I$ and $ q \in \mathcal{I}_{J'}^{I'}$, which implies $\alpha_J^I \leq \alpha_{J'}^{I'}$.    
    Consider the diagram at left of Fig.~\ref{fig:DgmsForlossInequalitiesForHatP} with element chase given on the right. 
    The parallelogram formed only with versions of $F$ commutes by 
    Eq.~\eqref{eq_F_commute}. 
    A similar argument gives that the parallelogram formed by versions of $G$ also commutes. 
    The left and right parallelograms, namely 
    \begin{equation*}
    \begin{tikzcd}
    \hat{F}(\alpha_J^I) \ar[r, "{\phi_{\alpha_{J}^{I}}}"] \ar[d]  & \hat G(\cT_\e\alpha_J^I) \ar[d] \\
    F(p) \ar[r,"\phi_p"] & G(\cT_\e p)
    \end{tikzcd} 
    \hspace{24pt} \text{and} \hspace{24pt} 
    \begin{tikzcd}
    \hat{F}(\alpha_{J'}^{I'}) \ar[r, "{\phi_{\alpha_{J'}^{I'}}}"] \ar[d]  & \hat G(\cT_\e\alpha_{J'}^{I'}) \ar[d] \\
    F(q) \ar[r, "\phi_q"] & G(\cT_\e q)
    \end{tikzcd} 
    \end{equation*}
    commute by Eqn.~\eqref{eq:phi_commute} when 
    $\phi_{\alpha_J^I}$ and $\phi_{\alpha_{J'}^{I'}}$ are defined. 
    The two remaining parallelograms are $\Parallelograml_\phi(p,q)$ and $\Parallelograml_\phi(\alpha_{J},\alpha_{J'})$, which we call the back and front respectively.
    
    Then, consider $a' \in F(p)$ and $a = F[\alpha_{J} \leq p]^{-1}(a')$ and follow respectively their images on the back and front parallelograms as labeled in the right of Fig.~\ref{fig:DgmsForlossInequalitiesForHatP}.
    We have that $d_{\T_\e q}^G (c,c') \leq  d_{\T_\e \alpha_{J'}}^{\hat{G}} (b,b')$ by Lemma \ref{lemma_preimage}, so by definition of the loss this implies that 
    \[
        \Lpl^{p,q}(\phi) \leq \Lpl^{\alpha_{J},\alpha_{J'}}(\phi)
    \]
    as required.

    Now, let us bound the triangle loss. 
    Consider the diagram 
    \[\begin{tikzcd}[cramped]
	\hat{F}(\alpha_{J}) && F(p) && \hat{F}(\alpha_{J'}) && F(p') \\
	&& \hat{G}( \T_\e \alpha_{J}) && G(\T_\e p)
	\arrow["{F[\alpha_{J}\leq p]}", from=1-1, to=1-3]
	\arrow["{\phi_{\alpha_{J}}}"'{pos=0.3}, from=1-1, to=2-3]
	\arrow["{F[  \leq ]}", from=1-3, to=1-5]
	\arrow["{\phi_p}"'{pos=0.3}, from=1-3, to=2-5]
	\arrow["{F[ \alpha_{J'} \leq p' ]}", from=1-5, to=1-7]
	\arrow["{\psi_{\T_\e \alpha_{J}}}"'{pos=0.8}, shift right=2, from=2-3, to=1-5]
	\arrow[shorten <=5pt, from=2-3, to=2-5]
	\arrow["{\psi_{\T_\e p}}"'{pos=0.7}, shift right=2, shorten >=5pt, from=2-5, to=1-7]
\end{tikzcd}\]
where we have set $p' = \T_\e \T_\e p$ and $\alpha_{J}' = \T_\e \T_\e \alpha_{J}$.
The left parallelogram 
\begin{equation*}
\begin{tikzcd}
	\hat{F}(\alpha_{J}) && F(p) \\
	&& \hat{G}( \T_\e \alpha_{J}) && G(\T_\e p)
	\arrow[ from=1-1, to=1-3]
	\arrow[ from=1-1, to=2-3]
	\arrow[ from=1-3, to=2-5]
	\arrow[from=2-3, to=2-5]
\end{tikzcd}
\end{equation*}
commutes by Eqn.~\eqref{eq:phi_commute}, and the right parallelogram 
    \[\begin{tikzcd}[cramped]
	&&\hat{F}(\alpha_{J}') \ar[rr] && F(p') \\
	\hat{G}( \T_\e \alpha_{J}) \ar[rr] \ar[urr] && G(\T_\e p) \ar[urr]
\end{tikzcd}\]
commutes by Eqn.~\eqref{eq:psi_commute}. 
Additionally, $F[ \alpha^{I}_{J} \leq  p]$ is isomorphism because for all $p \leq q$ in $\mathcal{I}^I_J \subset S_I$ the morphism $F[p \leq q]$ is an isomorphism. Then, by universal property it follows that $F[ \alpha^{I}_{J} \leq  p]$ is isomorphism as well.
Since all conditions of Lemma \ref{lemma_ll_triang}(1) are satisfied we can conclude that
\[
    \Ltd^p(\phi,\psi) \leq \Ltd^{\alpha_{J}}(\phi,\psi)
\]
finishing the proof.
\end{proof}

Analogously, we can define the diagrams 
$\Parallelogramr_\psi(\beta_{I}^{J},\beta_{I'}^{J'})$ and 
$\triangleu_{\phi,\psi}(\beta_{I}^{J}) $ 
with the corresponding losses 
$\Lpr^{\beta_{I}^{J},\beta_{I'}^{J'}}(\psi)$ and 
$\Ltu^{\beta_{I}^{J}}(\phi,\psi )$. 
Then a symmetric argument proves that for $p \leq q $ with $p \in \mathcal{J}_{I}^{J}$ and $q \in \mathcal{J}_{I}^{J}$, it holds that
\begin{equation*}
        \Lpr^{p,q}(\psi) \leq \Lpr^{\beta_{I}^{J},\beta_{I'}^{J'}}(\psi)
     \qquad \text{and} \qquad
      \Ltu^p(\phi,\psi) \leq \Ltu^{\beta_{I}^J}(\phi,\psi ).
\end{equation*}

Putting all of this infrastructure together and following the same argument as Thm.~\ref{thm:ControlledLossLine}, we have the final main theorem. 

\begin{theorem}
    \label{th_loss_Tk}
  Let $F$,$G$ in $C^P$ be cubical tame and $(\phi,\psi)$ be a given constructible, line-preserving $\T_\e$-assignment between them. 
  Then, the loss is given by 
    \begin{equation*}
    L(\phi,\psi) 
    = 
    \max \left\{  
    \max_{\hat{p} \prec \hat{q} \; \in \hat{B}} \{\Lpl^{\hat{p},\hat{q}}(\phi)\},  
    \max_{\hat{p} \in \{ \alpha_{J}^{I} \}} \{\Ltd^{\hat{p}}(\phi,\psi ) \}, 
    \max_{\tilde{p} \prec \tilde{q} \; \in \tilde{B}} \{\Lpr^{\tilde{p},\tilde{q}}(\psi)\} ,
    \max_{\tilde{p} \in \{ \beta_{J}^{I} \}} \{ \Ltu^{\tilde{p}}(\phi,\psi ) \} \right\}
    \end{equation*}
where 
$$\hat{B} = \{ \alpha_{J}^{I} \mid \alpha_{J}^{I} \text{ minimal in the poset } \{\alpha_{J'}^{I} \mid \mathcal{I}_{J'}^I \neq \emptyset, J' \in \Lambda \} \}$$
and 
$$\tilde{B} = \{ \beta_{I}^{J} \mid \beta_{I}^{J} \text{ minimal in the poset } \{\beta_{I'}^{J} \mid \mathcal{J}_{I'}^J \neq \emptyset, I' \in \Lambda'\} \}.$$ 
\end{theorem}

Finally, we use this theorem to compute the loss for multiparameter persistence modules. 
Consider $F,G: P^k \to \Vec_{\mathbb{F}}$ to be cubical tame and $(\phi,\psi)$ a constructible $\T_\e$-assignment. 
First we set $E$ to be the maximum  of the number of $i$-critical coordinates of $F$ and $G$ over all indices $i$.
Then, we define $N$ to be the maximum difference in rank for all pairs of maps following around two sides of the relevant parallelogram and triangle diagrams for the loss function. 
Explicitly, this is
{\small
  \begin{align*}
      N = \max \Bigg\{ 
      &\max_{\hat{p} \prec \hat{q} \; \in \hat{B}} 
        \left\{\rk (\phi_{\hat{q}} \circ \hat F[\hat{p} \leq \hat{q}] - \hat G[\T_\e \hat{p} \leq \T_\e\hat{q}] \circ \phi_{\hat{p}}) \right\},  
      \max_{\hat{p} \in \{ \alpha_{J}^{I} \}} 
        \left\{\rk (\psi_{\T_\e \hat{p}} \phi_{\hat{p}} - \hat F[\hat{p} \leq \T_\e \T_\e \hat{p}] )   \right\},  \\
     &  \max_{\tilde{p} \prec \tilde{q} \; \in \tilde{B}} 
        \left\{\rk (\psi_{\tilde{q}} \circ \tilde G[\tilde{p} \leq \tilde{q}] -  \tilde F[\T_\e \tilde{p} \leq \T_\e \tilde{q}] \circ \psi_{\tilde{p}}) \right\},  
    \max_{\tilde{p} \in \{ {\beta}_{I}^{J} \}} 
        \left\{\rk (\phi_{\T_\e \tilde{p}}\psi_{\tilde{p}} - \tilde G[\tilde {p} \leq \T_\e \T_\e \tilde{p}] ) \right\}. 
     \Bigg\}
  \end{align*}
}
We then define $M$ as a bound for the number of reducing constants necessary to check, given explicitly as 
  \begin{equation*}
    M = \max \left\{ 
      \max_{\hat{p} \in \{ \alpha_{J}^{I} \}} 
        \left\{ D^{\hat G}_{\T_\e \hat{p}} , D^{\hat F}_{\T_\e \T_\e \hat{p}}\right\} , 
      \max_{\hat{p} \in \{ \beta_{I}^{J} \}} 
        \left\{ D^{\tilde F}_{\T_\e \tilde{p}} , D^{\tilde G}_{\T_\e \T_\e \tilde {p}}\right\} 
    \right\}.
  \end{equation*}
With these constants, we can bound the time to calculate the loss as follows.

\begin{corollary}
  \label{coro_vec_Tk}
  Consider $F,G: P^k \to Vec_{\mathbb{F}}$ to be $q$-tame and cubical tame and $(\phi,\psi)$ a constructible $\T_\e$-assignment. The time complexity of calculating $L(\phi,\psi)$ is 
  $$O(E^{2k} N^3\log(M))$$ 
  with $E$, $N$, and $M$ given as above.   
  \end{corollary}
  \begin{proof}
      From Algorithm \ref{alg:lpl_calculation} we know that the time required to calculate $\Lpl^{p,q}(\phi)$ is $O( n^3\log(m))$ where \mbox{$n = \rk (\phi_q \circ  F[p \leq q]-  G[\T_\e p \leq \T_\e q] \circ \phi_p)$} and $m = |D^G_{\T_\e p}|$. 
      By construction $n \leq N$ and $m \leq M$.
      We have a similar expression for $\Lpr^{p,q}(\phi)$. 
      An analogous argument gives that the time required to calculate 
      $\Ltd^{p}(\phi,\psi)$ 
      is $O( n'^3\log(m'))$ 
      where \mbox{$n' = \rk (F[p \leq \T_\e \T_\e p]-  \psi_{\T_\e p} \circ \phi_p)$\}} 
      and $m' = |D^F_{\T_\e \T_\e p}|$. 
      Again $n' \leq N$ and $m' \leq M$.
      Likewise, we have a expression for $\Ltu^{p}(\phi,\psi)$.
      Thus, $O(N^3 log (M))$ bounds the required time to calculate the loss of any diagram. 
      Finally, by Thm.~\ref{th_loss_Tk}, the number of  diagrams that we need to check is bounded by 
      $2|\{\alpha_{J}^{I}\}| + 2|\{{\beta}_{I}^{J}\}| $, and each cardinality is upper bounded by $ (2E+1)^{2k} $. 
  \end{proof}

\section{Discussion}

\label{sec_discussion}

In this paper, in the context of persistence modules valued in concrete categories indexed by a poset, we defined a loss function that measures how far an assignment (i.e.~family of maps that resemble a natural transformation) is from being an interleaving. 
We showed that this loss provides an upper bound for the interleaving distance between generalized persistence modules. 
Then, we optimized the computation of the loss for persistence modules indexed by $P$ pr $P^k$ when $P$ is a complete linear order. 
Moreover, we gave polynomial time algorithms for the computation of the loss for persistence modules valued in categories whose objects have underlying finite sets or the category of vector spaces. 

Our results open the possibility of an algorithmic approach for obtaining upper bounds on the interleaving distance for commonly used constructions, such as multiparameter persistence modules, for which the exact computation of the interleaving distance has been shown to be NP-complete \citep{Bjerkevik2018}. 
However, we emphasize that our approach does not guarantee a tight bound. 
Hence, a natural direction for future work is the implementation of algorithms that search for the best possible bound over all possible assignments. 
While this problem will inevitably be NP-hard, additional care in developing strategies to explore the space of assignments might yield experimentally practical results, as was done in the case of this loss function for mapper graphs using linear programming \citep{Chambers2025}.

Furthermore, our framework applies to any structure that can be represented as a generalized persistence module, including merge trees \citep{Morozov2013}, Reeb graphs \citep{deSilva2016}, Mapper graphs \citep{2024bounding}, sheaves and cosheaves in concrete categories.
Consequently, the proposed loss yields an upper bound on the interleaving distance between such structures. 
The development of algorithms to compute this loss requires further work and is inherently context dependent, relying on the specific properties of each structure; for example, the existence of basis elements when the indexing poset consists of collections of open sets \citep{2024bounding}.

There are some immediate generalizations for our work which would be interesting for future work. 
One can replace  the notion of $\T_\e$-interleaving by weak $\T_\e$-interleaving \citep{Flow} and use a more general metric. 
Another route is exploring the interleaving distance of persistence modules defined in non-concrete such as the category of topological spaces. 
The definition of our loss function is based on a distance between elements inside an object of the category. 
Thus, the key is finding a definition of a loss function that is independent of the content of the object. 
Finally, we optimize the computation of the loss for persistence modules indexed by $P^k$ when $P$ is a complete linear order, due to the applications to multiparameter persistence that we were interested in. 
In future work, we would like to explore such optimization when the poset $P$ has other properties like modularity and latticity. 
For instance, sheaves and cosheaves are naturally indexed by distributive lattices.

\paragraph{Acknowledgements}
EM and AO are grateful to Peter Bubenik, Vin de Silva, and H{\aa}vard Bjerkevik for helpful discussions.
This work was supported in part by the National Science Foundation through grants CCF-2142713 and CCF-2106578.

\printbibliography

\end{document}